\documentclass[11pt,twoside]{article}

\usepackage{amsmath, amsthm,amsfonts,amssymb,mathtools,mathdots,scalerel,mathrsfs,stmaryrd,cmll}
\usepackage[utf8]{inputenc}
\usepackage[T1]{fontenc}
\usepackage{mlmodern, tikz-cd, quiver}
\usepackage[many]{tcolorbox}

\usepackage{enumitem, imakeidx, tocloft, csquotes, lipsum, tabularx}
\usepackage[hidelinks]{hyperref}
\usepackage{tikz, fancyhdr, xparse, xcolor}
\usepackage[british]{babel}
\usepackage[a4paper, top=4.5cm, bottom=4.5cm, left=4cm, right=4cm, asymmetric]{geometry}
\usepackage[textwidth=3cm, textsize=small, colorinlistoftodos]{todonotes}

\usepackage{crossreftools}

\usepackage{ahtitle, titlesec}
\usepackage{etoolbox}
\makeatletter
\patchcmd{\ttlh@hang}{\parindent\z@}{\parindent\z@\leavevmode}{}{}
\patchcmd{\ttlh@hang}{\noindent}{}{}{}
\makeatother

\newcommand\eqdef{\coloneqq}
\newcommand\nbd{\nobreakdash-\hspace{0pt}}
\newcommand\idd[1]{\mathrm{id}_{#1}}
\newcommand\after{\circ}

\newcommand\incl{\hookrightarrow}
\newcommand\incliso{\stackrel{\sim}{\hookrightarrow}}
\newcommand\iso{\stackrel{\sim}{\rightarrow}}
\newcommand\restr[2]{{#1}{\raisebox{0pt}{$|_{#2}$}}}

\newcommand\set[1]{\left\{ {#1} \right\}}
\newcommand\size[1]{\left|{#1}\right|}
\newcommand\isocl[1]{[#1]}

\newcommand{\overbar}[1]{{\mkern 1.5mu\overline{\mkern-1.5mu#1\mkern-1.5mu}\mkern 1.5mu}}

\newcommand\posnat{\mathbb{N} \setminus \set{0}}

\newcommand\slice[2]{{#1}/{\raisebox{-2pt}{$#2$}}}
\newcommand\join{\,{\star}\,}

\newcommand\optot[1]{#1^\circ}
\newcommand\dual[2]{\fun{D}_{#1}{#2}}

\newcommand\cat[1]{\mathbf{#1}}
\newcommand\smcat[1]{\mathscr{#1}}

\newcommand\fun[1]{\mathsf{#1}}
\newcommand\property[1]{\mathrm{#1}}

\newcommand\order[2]{#2^{(#1)}}

\newcommand\ogpos{\cat{ogPos}}
\newcommand\ogposbot{\ogpos^+}

\newcommand\poscat{\cat{Pos}}

\newcommand\omegacat{\omega\cat{Cat}}
\newcommand\rdcpx{\cat{RDCpx}}

\newcommand\dchaug{\cat{DCh}^+}
\newcommand\stcpx{\cat{DCh}^+_{\mathit{St}}}
\newcommand\sstcpx{\cat{DCh}^+_{\mathit{sSt}}}

\newcommand\hasseo[1]{\vec{\mathscr{H}}{#1}}

\DeclareMathOperator{\clos}{cl}
\newcommand\clset[1]{\mathrm{cl}\set{#1}}

\newcommand\codim[2]{\mathrm{codim}_{#2}(#1)}
\newcommand\maxel[1]{\mathscr{M}\!\mathit{ax}\,#1}
\newcommand\grade[2]{#2_{#1}}
\DeclareMathOperator{\lydim}{lydim}
\DeclareMathOperator{\frdim}{frdim}

\newcommand\bound[2]{\partial_{#1}^{#2}}
\newcommand\faces[2]{\Delta_{#1}^{#2}}
\newcommand\cofaces[2]{\nabla_{#1}^{#2}}
\newcommand\inter[1]{\mathrm{int}\,#1}

\newcommand\cp[1]{\,{\scriptstyle\#}_{#1}\,}

\newcommand\cpiso[2]{\,{\scriptstyle\#}^{#2}_{#1}\,}
\newcommand\celto{\Rightarrow}

\newcommand\submol{\sqsubseteq}

\newcommand\flow[2]{\mathscr{F}_{#1}{#2}}
\newcommand\extflow[2]{\overbar{\mathscr{F}}_{#1}{#2}}
\newcommand\maxflow[2]{\mathscr{M}_{#1}{#2}}

\newcommand\layerings[2]{\mathscr{L}\!\mathit{ay}_{#1}{#2}}
\newcommand\orderings[2]{\mathscr{O}\!\mathit{rd}_{#1}{#2}}
\newcommand\lto[2]{\fun{o}_{#1, #2}}

\newcommand\gray{\otimes}
\newcommand\sus[1]{\fun{S}{#1}}

\newcommand\augm[1]{{{#1}_\bot}}
\newcommand\dimin[1]{{#1}_{\not\bot}}

\newcommand\gener[1]{\mathscr{#1}}
\newcommand\cwcom[2]{({#1},\gener{#2})}

\newcommand\chain[2]{{#1}_{#2}}
\newcommand\der{\mathrm{d}}
\newcommand\eau{\mathrm{e}}
\newcommand\freeab[1]{\mathbb{Z}{#1}}

\newcommand\dir[1]{{#1}^{\rightarrow}}
\newcommand\dfreeab[1]{\vec{\mathbb{Z}}{#1}}
\newcommand\freemon[1]{\mathbb{N}{#1}}

\newcommand\linea[1]{\fun{\lambda}{#1}}
\newcommand\nufun[1]{\fun{\nu}{#1}}
\newcommand\spanset[1]{\langle {#1} \rangle}
\DeclareMathOperator{\Ima}{Im}
\DeclareMathOperator{\supp}{supp}

\newcommand\gltab[3]{{#3}_{#1}^{#2}}
\newcommand\batom[1]{\langle #1 \rangle}

\newcommand\molec{\mathit{Mol}}
\newcommand\molecin[1]{\slice{\molec}{#1}}
\newcommand\atom{\mathit{Atom}}
\newcommand\atomin[1]{\slice{\atom}{#1}}

\newcommand\omegatit{\texorpdfstring{$\omega$}{omega}}

\newcommand\skel[2]{\sigma_{\leq {#1}}#2}
\newcommand\cpable[1]{\times_{#1}}
\DeclareMathOperator{\cspan}{span}

\newtheoremstyle{ittheorem}
  {\topsep}   
  {\topsep}   
  {\itshape}  
  {0pt}       
  {\sffamily \itshape \bfseries} 
  { ---}         
  {5pt plus 1pt minus 1pt} 
  {}          

\newtheoremstyle{itdfn}
  {\topsep}   
  {\topsep}   
  {}  
  {0pt}       
  {\sffamily \itshape \bfseries} 
  {}         
  {5pt plus 1pt minus 1pt} 
  {\thmnumber{#2}{\thmnote{\normalfont\ \ %
  {\sffamily(#3)}.}}}          

\newtheoremstyle{itrmk}
  {0.5\topsep}   
  {0.5\topsep}   
  {\normalfont}  
  {0pt}       
  {\sffamily \itshape} 
  { ---}         
  {5pt plus 1pt minus 1pt} 
  {}          
  
\newtheoremstyle{itexm}
  {0.5\topsep}      
  {0.5\topsep}      
  {\normalfont}     
  {0pt}             
  {\sffamily \itshape \bfseries \color{\mycolor}}        
  {\\}               
  {5pt plus 1pt minus 1pt} 
  {\thmname{#1} \thmnumber{#2}{\thmnote{\normalfont\ \ %
  {\sffamily(#3)}.}}}           
  
\makeatletter
  \renewcommand\@upn{\textit}
\makeatother

\newcommand\mycolor{black}

\tcbset{
    boxstyle/.style={%
        colframe=\mycolor, colback=gray!10,%
        left=2mm, right=2mm,%
        sharp corners,%
        breakable=true,%
        borderline west={1pt}{0pt}{\mycolor},%
        enhanced, boxrule=0pt, frame hidden}
    }
    
\theoremstyle{ittheorem}
\newtheorem{thm}{Theorem}[section]
\newtheorem{prop}[thm]{Proposition}
\newtheorem{cor}[thm]{Corollary}
\newtheorem{lem}[thm]{Lemma}

\theoremstyle{itdfn}
\newtheorem{dfn}[thm]{}
\theoremstyle{itrmk}
\newtheorem{rmk}[thm]{Remark}
\newtheorem{comm}[thm]{Comment}
\newtheorem{exm}[thm]{Example}

\setlist{leftmargin=20pt,itemsep=0pt,topsep=1ex}

\titleformat{\section}
 {\large\scshape}{\thesection.}{1em}{}

\titleformat{\subsection}
 {\normalsize\itshape}{\thesubsection.}{1em}{}
\titlespacing*{\subsection}
{0pt}{1.5ex plus 1ex minus .2ex}{1.5ex plus .2ex}

\renewcommand{\cftsecpagefont}{\mdseries}

\makeatletter \renewcommand{\cftsecfillnum}[1]{%
  {\cftsecleader}\nobreak
  \makebox[\@pnumwidth][\cftpnumalign]{\cftsecpagefont \oldstylenums{#1}}\cftsecafterpnum\par
} \makeatother

\newcommand\runtitle{acyclicity conditions on pasting diagrams}
\newcommand\runauthor{hadzihasanovic and kessler}

\title{Acyclicity conditions on pasting diagrams}

\author{Amar Hadzihasanovic and Diana Kessler}

\institution{Tallinn University of Technology}

\begin{document}

\thispagestyle{empty}
\maketitle 

\noindent\makebox[\textwidth][r]{%
	\begin{minipage}[t]{.7\textwidth}
\small \emph{Abstract.}
We study various acyclicity conditions on higher-categorical pasting diagrams in the combinatorial framework of regular directed complexes. 
We present an apparently weakest acyclicity condition under which the $\omega$\nbd category presented by a diagram shape is freely generated in the sense of polygraphs.
We then consider stronger conditions under which this $\omega$\nbd category is equivalent to one obtained from an augmented directed chain complex in the sense of Steiner, or consists only of subsets of cells in the diagram.
Finally, we study the stability of these conditions under the operations of pasting, suspensions, Gray products, joins and duals.
	\end{minipage}}

\vspace{20pt}

\makeaftertitle

\normalsize

\noindent\makebox[\textwidth][c]{%
\begin{minipage}[t]{.75\textwidth}
\setcounter{tocdepth}{1}
\tableofcontents
\end{minipage}}

\section*{Introduction}

Pasting diagrams are a central tool for studying the composition of cells in higher\nbd dimensional categories.
The notion of 2\nbd categorical pasting was introduced by B\'enabou \cite{benabou67}; in the 1980s and 1990s, a number of frameworks for $n$\nbd categorical pasting emerged, with corresponding \emph{pasting theorems} guaranteeing that a pasting diagram admits a suitably unique composite
\cite{johnson1989combinatorics, power1991pasting, street1991parity, steiner1993algebra}, see \cite{forest2022unifying} for a recent survey.

A pasting diagram is, informally, a \emph{composable} configuration of cells in an $n$\nbd category, such as the following:
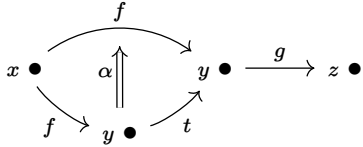
\[
\begin{tikzcd}[sep=small]
	{{\scriptstyle x}\;\bullet} && {{\scriptstyle y}\;\bullet} && {{\scriptstyle z}\;\bullet} \\
	& {{\scriptstyle y}\;\bullet}
	\arrow[""{name=0, anchor=center, inner sep=0}, "f", curve={height=-18pt}, from=1-1, to=1-3]
	\arrow["g", from=1-3, to=1-5]
	\arrow["f"', curve={height=6pt}, from=1-1, to=2-2]
	\arrow["t"', curve={height=6pt}, from=2-2, to=1-3]
	\arrow["\alpha", shorten <=3pt, shorten >=6pt, Rightarrow, from=2-2, to=0]
\end{tikzcd}
\]
More in general, one considers ``non-pasting'' diagram shapes, that do not admit a composite, such as the following:
\[
\begin{tikzcd}
     {\bullet} & {\bullet} & {\bullet}
     \arrow[from=1-2, to=1-1]
     \arrow[from=1-2, to=1-3]
\end{tikzcd}
\]
Various formalisms for diagrams have tried to encode, either \emph{combinatorially} or \emph{topologically} the information of such a diagram, in a way that reflects (and generalises in higher dimensions) the content of these pictures, and has a univocal interpretation, in the form of a higher category \emph{presented} by the diagram shape, together with a functor out of it.

Most of these formalisms include, as part of the definition of a pasting diagram shape, some \emph{acyclicity} conditions of varying strength, barring --- at the very least --- the existence of ``direct loops'' where a cell may appear more than one time in a composite.
These conditions serve at least three purposes:
\begin{enumerate}
	\item to restrict the class of admissible structures so that some undesirable examples are not part of it;
	\item to guarantee that an $n$\nbd category can be formed out of \emph{subdiagrams}, or ``composable subsets'' of cells in the diagram;
	\item to ensure that the presented $n$\nbd category is \emph{freely generated} in the sense of polygraphs or computads \cite{street1976limits, burroni1993higher, ara2023polygraphs}.
\end{enumerate}
On the other hand, imposing such conditions comes with a cost.
Firstly, they exclude commonly occuring shapes of pasting diagrams that appear in dimension 3: for example, a 3\nbd cell of the form
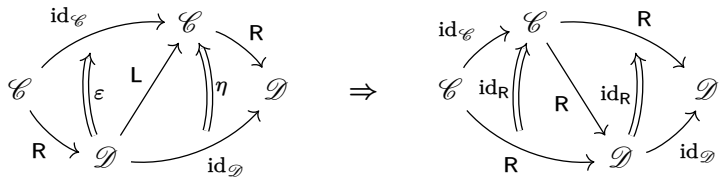
\[
	\begin{tikzcd}[sep=small]
	&& {\smcat{C}} &&&& {\smcat{C}} \\
	{\smcat{C}} &&& {\smcat{D}} & {\Rightarrow} & {\smcat{C}} &&& {\smcat{D}} \\
	& {\smcat{D}} &&&&&& {\smcat{D}}
	\arrow["{\fun{R}}"', curve={height=6pt}, from=2-1, to=3-2]
	\arrow["{\fun{L}}", from=3-2, to=1-3]
	\arrow[""{name=0, anchor=center, inner sep=0}, "{\idd{\smcat{C}}}", curve={height=-12pt}, from=2-1, to=1-3]
	\arrow[""{name=1, anchor=center, inner sep=0}, "{\idd{\smcat{D}}}"', curve={height=12pt}, from=3-2, to=2-4]
	\arrow["{\fun{R}}", curve={height=-6pt}, from=1-3, to=2-4]
	\arrow["{\idd{\smcat{C}}}", curve={height=-6pt}, from=2-6, to=1-7]
	\arrow["{\fun{R}}"', from=1-7, to=3-8]
	\arrow[""{name=2, anchor=center, inner sep=0}, "{\fun{R}}", curve={height=-12pt}, from=1-7, to=2-9]
	\arrow[""{name=3, anchor=center, inner sep=0}, "{\fun{R}}"', curve={height=12pt}, from=2-6, to=3-8]
	\arrow["{\idd{\smcat{D}}}"', curve={height=6pt}, from=3-8, to=2-9]
	\arrow["\varepsilon"', curve={height=-6pt}, shorten >=7pt, Rightarrow, from=3-2, to=0]
	\arrow["\eta"', curve={height=6pt}, shorten <=7pt, Rightarrow, from=1, to=1-3]
	\arrow["{\idd{\fun{R}}}", curve={height=-6pt}, shorten <=7pt, Rightarrow, from=3, to=1-7]
	\arrow["{\idd{\fun{R}}}", curve={height=6pt}, shorten >=7pt, Rightarrow, from=3-8, to=2]
\end{tikzcd}\]
appears as a ``weakened'' form of one of the \emph{triangle equations} in the theory of pseudoadjunctions of 2-categories, but its shape is \emph{not} acyclic, not even in a weak sense, due to the 1\nbd cells in the interiors of the two sides forming a direct loop.
If we look at non-pasting diagrams, simple counterexamples appear already in dimension 1, where looping diagrams of shape 
\[
\begin{tikzcd}
    {\bullet} & {\bullet}
    \arrow[curve={height=7pt}, from=1-1, to=1-2]
    \arrow[curve={height=7pt}, from=1-2, to=1-1]
\end{tikzcd}
\]
are perfectly well-defined, yet any result proved with an assumption of acyclicity will not extend to them.
Secondly, acyclicity properties are \emph{global} properties that tend to be unstable, not preserved under common operations. 
For example, stronger acyclicity conditions are not preserved under direction-reversing duality operations, and weaker acyclicity conditions are not preserved under pasting and various forms of products.

In \cite{hadzihasanovic2020combinatorial, hadzihasanovic2020diagrammatic}, the first-named author started exploring a framework for diagrams inspired by Steiner's approach in \cite{steiner1993algebra}, but based on the ``local'' property of \emph{regularity}, which requires the input and output boundaries of cells in a diagram to be closed balls in a topological sense, and is remarkably stable under all sorts of constructions.
This has resulted, recently, in the book-length exposition \cite{hadzihasanovic2024combinatorics}.
The structures encoding ``diagram shapes'' in this framework are called \emph{regular directed complexes}.

The inductive definition of pasting diagram shapes, called \emph{molecules} in this framework, ensures that ``bad'' examples are left out, but is general enough in the sense that it allows further shapes that one might wish to have.
However, in general, the two other properties listed above --- subdiagrams form an $n$\nbd category, the $n$\nbd category is a polygraph --- are no longer guaranteed.
In this article we use regular directed complexes as a backdrop for a more refined study of acyclicity conditions, and their role in achieving these properties.

For the first, we argue that, to a certain extent, it is a non-problem in that one can replace \emph{subsets} with more general \emph{morphisms} whose domains are molecules in order to obtain an $\omega$\nbd category even without acyclicity.
This is akin to the situation with directed graphs, where \emph{linear subgraphs} only form a category if the graph is acyclic, but \emph{paths} always form a category.

For the ``free generation'' property, we show that this is achieved by a weaker notion of acyclicity, called \emph{frame-acyclicity}, which is shared by all regular directed complexes of dimension lower or equal then 3 --- including the non-acyclic examples above.
Frame-acyclicity for molecules is equivalent to \emph{splitness} in the sense of \cite{steiner1993algebra}, and also has interesting algorithmic consequences as studied by the authors in 
\cite{hadzihasanovic2023higher}. 
However, frame-acyclicity is very technical, difficult to check, and its stability properties are unclear. so it is useful to consider other conditions which are easier to check in practice, but more restrictive.

The first such notion is \emph{dimension-wise acyclicity} which is tied to Steiner's loop-freeness property in the theory of \emph{augmented directed chain complexes} \cite{steiner2004omega}, and allows us to make a precise connection between this and our framework.
In particular, we prove the existence of an isomorphism between the $\omega$-category presented by a dimension-wise acyclic regular directed complex, and the $\omega$-category obtained by first passing to an augmented directed chain complex and then applying Steiner's functor $\nufun{}$.
A slightly stronger notion, that we call \emph{strong dimension-wise acyclicity}, is what guarantees that every morphism from a molecule is injective, hence ``subsets suffice''.
These two conditions have the nice property of being closed under direction-reversing duals, but not under other operations such as pasting (for molecules), Gray products, or joins.
For this reason, we finally consider an even stronger notion, \emph{acyclicity}, corresponding to total loop-freeness in \cite{steiner1993algebra}, and which is stable under the latter operations.

\subsection*{Structure of the article}
In Section \ref{sec:regular}, we introduce oriented graded posets, our basic data structure on which we define regular directed complexes and the inductive class of molecules, together with their category $\ogpos$.
We then present the construction of a strict $\omega$-category $\molecin{P}$ from an oriented graded poset $P$, whose cells are morphisms from a molecule to $P$, taken up to isomorphism in the slice category $\slice{\ogpos}{P}$.
Section \ref{sec:layerings} is dedicated to the theory of layerings of molecules, which are ways of writing a molecule as a pasting decomposition in which each term or layer contains exactly one maximal element of dimension greater then the pasting dimension. 
In Section \ref{sec:frameacy}, we discuss the frame-acyclicity condition, giving full proofs of some results outlined in \cite{hadzihasanovic2023higher}. 
We end the section with the proof that, if an oriented graded poset $P$ has frame-acyclic molecules, then $\molecin{P}$ is a polygraph.
In Section \ref{sec:dwacy} we make the connection between our framework and Steiner's theory of augmented directed chain complexes.
We prove that for a dimension-wise acyclic regular directed complex $P$, the $\omega$\nbd category $\molecin{P}$ is isomorphic to the $\omega$-category obtained by applying Steiner's $\nufun{}$ functor to the Steiner complex obtained from $P$.
In Section \ref{sec:stronger} we study the stronger acyclicity conditions implying that the $\omega$-category $\molecin{P}$ consists only of subsets of $P$.
We also show that the strongest condition translates to the associated Steiner complex being a ``strong Steiner complex''.
Finally, in Section \ref{sec:stability} we study the stability of the acyclicity conditions presented above under the operations of pasting, suspensions, Gray products, joins and duals.

\subsection*{Note}
The content of this article was recently exposed, with more detail and all results reproved from scratch, in \cite[Chapter 8 and Chapter 11]{hadzihasanovic2024combinatorics}, as part of a reference book written by the first-named author.
The results, however, have been developed in cooperation by the two authors, and many of them have not appeared in print before.
The purpose of this article is both to give a clearer picture of the original research developments --- in contrast to the book, we do not reprove results when a proof with roughly the same content has appeared before, even when the definitions are slightly different --- and to offer a concise, self-contained treatment of a topic which seems particularly subtle and somewhat misunderstood in the theory of higher-categorical diagrams.

\subsection*{Acknowledgements}
The first-named author was supported by Estonian Research Council grant PSG764.
We thank Guillaume Laplante-Anfossi for discussions which helped shape the article and Cl\'emence Chanavat for discussions about \cite{steiner2004omega}.

\section{Regular directed complexes and \omegatit-categories} \label{sec:regular}

\noindent 
In this section, we give an overview of our combinatorial framework for higher-categorical diagrams, and state without proof some of the foundational results.
This is exposed in much more detail in the first chapters of \cite{hadzihasanovic2024combinatorics}.
The basic structure that we use to represent shapes of diagrams is called an \emph{oriented graded poset}. 
Before introducing it, we first recall some notions about posets with order relation $\leq$.

\begin{dfn}[Faces and cofaces]
Let $P$ be a poset.
Given elements $x, y \in P$, we say that $y$ \emph{covers} $x$ if $x < y$ and, for all $y' \in P$, if $x < y' \leq y$ then $y' = y$.
For each $x \in P$, the sets of \emph{faces} and \emph{cofaces} of $x$ are, respectively,
\[
	    \faces{}{} x \eqdef \set{ y \in P \mid \text{$x$ covers $y$} } \quad
	    \text{and} \quad
        	\cofaces{}{} x \eqdef \set{ y \in P \mid \text{$y$ covers $x$} }.
\]
\end{dfn}

\begin{dfn}[Closed subsets]
Let $P$ be a poset and $U \subseteq P$.
The \emph{closure of $U$} is the subset
$\clos{U} \eqdef \set{ x \in P \mid \text{there exists $y \in U$ such that $x \leq y$} }$.
We say that $U$ is \emph{closed} if $U = \clos{U}$.
\end{dfn}

\begin{dfn}[Graded poset]
	A poset $P$ is \emph{graded} if, for all $x \in P$, all maximal chains in $\clset{x}$ have the same finite size $n$.
	In this case, we let $\dim{x}$, the \emph{dimension} of $x$, be equal to $n - 1$.
	If $P$ is a graded poset, the \emph{dimension} of $P$ is 
\[    
	\dim P \eqdef
	\begin{cases}
		\max \left(\set{-1} \cup \set{ \dim x \mid x \in P }\right) & \text{if defined,} \\
		\infty & \text{otherwise.}
	\end{cases}
\]
	For each $n \in \mathbb{N}$, we write $\grade{n}{P} \eqdef \set{ x \in P \mid \dim x = n }$.
\end{dfn}

\begin{rmk}
	In a graded poset, if $y \in \faces{}{}x$, then $\dim{y} = \dim{x} - 1$.
\end{rmk}

\begin{dfn}[Oriented graded poset]
	An \emph{oriented graded poset} is a graded poset $P$ together with, for all $x \in P$, a bipartition $\faces{}{}x = \faces{}{-}x + \faces{}{+}x$ of the set of faces of $x$ into a set $\faces{}{-}x$ of \emph{input faces} and a set $\faces{}{+}x$ of \emph{output faces}.
\end{dfn}

\begin{rmk}
By duality, this induces a bipartition $\cofaces{}{}x = \cofaces{}{+}x + \cofaces{}{-}x$ of the set of cofaces of each element $x$.
\end{rmk}

\noindent
We will use $\alpha, \beta, \ldots$ for variables ranging over $\set{ +, - }$.
We let $-\alpha$ be $-$ if $\alpha = +$ and $+$ if $\alpha = -$.

\begin{dfn}[Oriented Hasse diagram]
Let $P$ be an oriented graded poset.
The \emph{oriented Hasse diagram of $P$} is the directed graph $\hasseo{P}$ whose
\begin{itemize}
    \item set of vertices is the underlying set of $P$, and
    \item set of edges is $\set{ (x, y) \mid \text{$x \in \faces{}{-}y$ or $y \in \faces{}{+}x$} }$, where the source of $(x, y)$ is $x$ and the target is $y$.
\end{itemize}
\end{dfn}

\noindent
The oriented Hasse diagram is the usual Hasse diagram of a poset, with edges representing input and output faces given opposite orientations (input from lower to higher dimension, and output from higher to lower dimension).
An oriented graded poset is uniquely specified by its oriented Hasse diagram, together with the $\dim{}$ function, which graphically can be encoded by height, as in the following example.

\begin{exm}
Consider the 2-dimensional pasting diagram shape
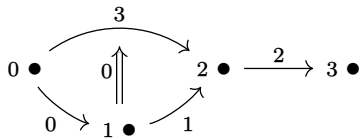
\[
\begin{tikzcd}[sep=small]
	{{\scriptstyle 0}\;\bullet} && {{\scriptstyle 2}\;\bullet} && {{\scriptstyle 3}\;\bullet} \\
	& {{\scriptstyle 1}\;\bullet}
	\arrow[""{name=0, anchor=center, inner sep=0}, "3", curve={height=-18pt}, from=1-1, to=1-3]
	\arrow["2", from=1-3, to=1-5]
	\arrow["0"', curve={height=6pt}, from=1-1, to=2-2]
	\arrow["1"', curve={height=6pt}, from=2-2, to=1-3]
	\arrow["0", shorten <=3pt, shorten >=6pt, Rightarrow, from=2-2, to=0]
\end{tikzcd}
\]
where we used progressive natural numbers for cells of each dimension.
This is encoded by the oriented graded poset whose oriented Hasse diagram is
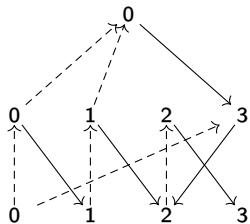
\[\begin{tikzpicture}[xscale=4, yscale=4, baseline={([yshift=-.5ex]current bounding box.center)}]
\draw[->, densely dashed, draw=\mycolor] (0.125, 0.19999999999999998) -- (0.125, 0.4666666666666667);
\draw[->, densely dashed, draw=\mycolor] (0.19999999999999998, 0.19999999999999998) -- (0.8, 0.4666666666666667);
\draw[->, densely dashed, draw=\mycolor] (0.375, 0.19999999999999998) -- (0.375, 0.4666666666666667);
\draw[->, densely dashed, draw=\mycolor] (0.625, 0.19999999999999998) -- (0.625, 0.4666666666666667);
\draw[->, draw=black] (0.15, 0.4666666666666667) -- (0.35, 0.19999999999999998);
\draw[->, densely dashed, draw=\mycolor] (0.16249999999999998, 0.5333333333333333) -- (0.4625, 0.7999999999999999);
\draw[->, draw=black] (0.4, 0.4666666666666667) -- (0.6, 0.19999999999999998);
\draw[->, densely dashed, draw=\mycolor] (0.3875, 0.5333333333333333) -- (0.4875, 0.7999999999999999);
\draw[->, draw=black] (0.65, 0.4666666666666667) -- (0.85, 0.19999999999999998);
\draw[->, draw=black] (0.85, 0.4666666666666667) -- (0.65, 0.19999999999999998);
\draw[->, draw=black] (0.5375, 0.7999999999999999) -- (0.8375, 0.5333333333333333);
\node[text=black, font={\scriptsize \sffamily}, xshift=0pt, yshift=0pt] at (0.125, 0.16666666666666666) {0};
\node[text=black, font={\scriptsize \sffamily}, xshift=0pt, yshift=0pt] at (0.375, 0.16666666666666666) {1};
\node[text=black, font={\scriptsize \sffamily}, xshift=0pt, yshift=0pt] at (0.625, 0.16666666666666666) {2};
\node[text=black, font={\scriptsize \sffamily}, xshift=0pt, yshift=0pt] at (0.875, 0.16666666666666666) {3};
\node[text=black, font={\scriptsize \sffamily}, xshift=0pt, yshift=0pt] at (0.125, 0.5) {0};
\node[text=black, font={\scriptsize \sffamily}, xshift=0pt, yshift=0pt] at (0.375, 0.5) {1};
\node[text=black, font={\scriptsize \sffamily}, xshift=0pt, yshift=0pt] at (0.625, 0.5) {2};
\node[text=black, font={\scriptsize \sffamily}, xshift=0pt, yshift=0pt] at (0.875, 0.5) {3};
\node[text=black, font={\scriptsize \sffamily}, xshift=0pt, yshift=0pt] at (0.5, 0.8333333333333333) {0};
\end{tikzpicture}\]
where we also (redundantly) represented the ``input'' edges as densely dashed lines for extra emphasis.
\end{exm}

\begin{dfn}[Morphism of oriented graded posets]
Let $P, Q$ be oriented graded posets.
A \emph{morphism} $f\colon P \to Q$ is a function of their underlying sets which, for all $x \in P$ and $\alpha \in \set{ +, - }$, induces a bijection between $\faces{}{\alpha} x$ and $\faces{}{\alpha} f(x)$.
An \emph{inclusion} of oriented graded posets is an injective morphism.

We let $\ogpos$ denote the category whose objects are oriented graded posets and morphisms are morphisms of oriented graded posets.
\end{dfn}

\noindent We list some basic properties of morphisms of oriented graded posets.
By \emph{closed map}, we mean a map that sends closed subsets to closed subsets.

\begin{lem} \label{lem:properties_of_morphisms}
Let $f\colon P \to Q$ be a morphism of oriented graded posets.
Then $f$ is an order-preserving, closed, dimension-preserving map of the underlying graded posets.

Moreover, if $f$ is an inclusion, then $f$ is order-reflecting, and reflects input and output faces, that is, if $f(x) \in \faces{}{\alpha}f(y)$, then $x \in \faces{}{\alpha}y$.

In particular, $f$ is an isomorphism if and only if it is a surjective inclusion.
\end{lem}

\begin{rmk}
Consequently, there is a forgetful functor $\fun{U}\colon \ogpos \to \poscat$, where $\poscat$ is the category of posets and order-preserving maps.
\end{rmk}

\begin{rmk}
The inclusion of a closed subset $U \subseteq P$ of an oriented graded poset with the induced order and orientation is always an inclusion of oriented graded posets.
\end{rmk}

\begin{prop} \label{prop:pushouts_in_ogpos}
The category $\ogpos$ has a strict initial object $\varnothing$ and pushouts of inclusions along inclusions, which are both preserved and reflected by $\fun{U}\colon \ogpos \to \poscat$.
Moreover,
\begin{enumerate}
	\item the pushout of an inclusion along an inclusion is an inclusion,
	\item a pushout square of inclusions is also a pullback square.
\end{enumerate}
\end{prop}

\begin{dfn}[Input and output $n$-boundaries]
Let $U$ be a closed subset of an oriented graded poset, and let $\maxel{U} \subseteq U$ be its subset of maximal elements.
For all $\alpha \in \set{ +, - }$ and $n \in \mathbb{N}$, let $\faces{n}{\alpha} U \eqdef
    \set{ x \in U_n \mid \cofaces{}{-\alpha} x \cap U = \varnothing  }$.
The \emph{input} and \emph{output $n$\nbd boundary} of $U$ are, respectively, the closed subsets
\[
    \bound{n}{-} U \eqdef \clos{(\faces{n}{-} U)} \cup \bigcup_{k < n} \clos{\grade{k}{(\maxel{U})}}, \quad
    \bound{n}{+} U \eqdef \clos{(\faces{n}{+} U)} \cup \bigcup_{k < n} \clos{\grade{k}{(\maxel{U})}}.
\]
We let $\bound{n}{} U \eqdef \bound{n}{+} U \cup \bound{n}{-} U$.
For $n < 0$, we let $\faces{n}{\alpha}U = \bound{n}{\alpha}U \eqdef \varnothing$.
\end{dfn}

\begin{lem} \label{lem:faces_intersection}
Let $V \subseteq U$ be closed subsets of an oriented graded poset, $n \in \mathbb{N}$, and $\alpha \in \set{ +, - }$.
Then $V \cap \faces{n}{\alpha}U \subseteq \faces{n}{\alpha}V$.
\end{lem}

\begin{prop} \label{prop:inclusions_preserve_faces}
Let $\imath\colon P \incl Q$ be an inclusion of oriented graded posets and $U \subseteq P$ a closed subset.
For all $n \in \mathbb{N}$ and $\alpha \in \set{ +, - }$, $\imath(\faces{n}{\alpha}U) = \faces{n}{\alpha} \imath(U)$ and $\imath(\grade{n}{ (\maxel{U}) }) = \grade{n}{ (\maxel{\imath(U)}) }$.
Consequently, $\imath(\bound{n}{\alpha}U) = \bound{n}{\alpha}\imath(U)$.
\end{prop}

\noindent 
Not all oriented graded posets represent shapes of diagrams.
In what follows, we introduce the two constructions that we use to build shapes of diagrams, then give the inductive construction of molecules, the subclass of oriented graded posets that represents well-formed shapes of diagrams.

\begin{dfn}[Pasting construction]
Let $U$, $V$ be oriented graded posets, $k \in \mathbb{N}$, and let $\varphi\colon \bound{k}{+}U \incliso \bound{k}{-}V$ be an isomorphism.
The \emph{pasting of $U$ and $V$ at the $k$\nbd boundary along $\varphi$} is the oriented graded poset $U \cpiso{k}{\varphi} V$ obtained in $\ogpos$ as the pushout
\[\begin{tikzcd}
	\bound{k}{+}U & \bound{k}{-}V & V \\
	U && U \cpiso{k}{\varphi} V.
	\arrow["\varphi", hook, from=1-1, to=1-2]
	\arrow[hook, from=1-2, to=1-3]
	\arrow[hook', from=1-1, to=2-1]
	\arrow["\imath_U", hook, from=2-1, to=2-3]
	\arrow["\imath_V", hook', from=1-3, to=2-3]
	\arrow["\lrcorner"{anchor=center, pos=0.125, rotate=180}, draw=none, from=2-3, to=1-1]
\end{tikzcd}\]
\end{dfn}

\begin{dfn}[Globularity]
Let $U$ be an oriented graded poset.
We say that $U$ is \emph{globular} if, for all $k, n \in \mathbb{N}$ and $\alpha, \beta \in \set{ +, - }$, if $k < n$ then
\begin{equation*}
    \bound{k}{\alpha}(\bound{n}{\beta}U) = \bound{k}{\alpha}U.
\end{equation*}
Let $U$ be a globular oriented graded poset such that $n \eqdef \dim{U} < \infty$.
For all $\alpha \in \set{+, -}$, we write $\bound{}{\alpha}U \eqdef \bound{n-1}{\alpha}U$, and $\bound{}{}U \eqdef \bound{}{+}U \cup \bound{}{-}U$.
We also write $\inter{U} \eqdef U \setminus \bound{}{} U$.
\end{dfn}

\begin{dfn}[Roundness]
Let $U$ be an oriented graded poset.
We say that $U$ is \emph{round} if it is globular and, for all $n < \dim{U}$,
\begin{equation*}
    \bound{n}{-}U \cap \bound{n}{+}U = \bound{n-1}{}U.
\end{equation*}
\end{dfn}

\begin{dfn}[Rewrite construction]
Let $U$, $V$ be round oriented graded posets of the same finite dimension $n$, and suppose $\varphi\colon \bound{}{}U \incliso \bound{}{}V$ is an isomorphism restricting to isomorphisms $\varphi^\alpha\colon \bound{}{\alpha} U \incliso \bound{}{\alpha} V$ for each $\alpha \in \set{ +, - }$.
Construct the pushout in $\ogpos$
\[\begin{tikzcd}
	\bound{}{}U & \bound{}{}V & V \\
	U && \bound{}{}(U \celto^\varphi V).
	\arrow["\varphi", hook, from=1-1, to=1-2]
	\arrow[hook, from=1-2, to=1-3]
	\arrow[hook', from=1-1, to=2-1]
	\arrow[hook, from=2-1, to=2-3]
	\arrow[hook', from=1-3, to=2-3]
	\arrow["\lrcorner"{anchor=center, pos=0.125, rotate=180}, draw=none, from=2-3, to=1-1]
\end{tikzcd}\]
The \emph{rewrite of $U$ into $V$ along $\varphi$} is the oriented graded poset $U \celto^\varphi V$ obtained by adjoining a single $(n+1)$\nbd dimensional element $\top$ to $\bound{}{}(U \celto^\varphi V)$, with
\begin{equation*}
    \faces{}{-}\top \eqdef U_n, \quad \quad \faces{}{+}\top \eqdef V_n.
\end{equation*}
\end{dfn}

\begin{dfn}[Point]
The \emph{point} is the oriented graded poset $1$ with a single element and trivial orientation.
\end{dfn}

\noindent 
We are now ready to give the definition of molecules.

\begin{dfn}[Molecules and atoms]
The class of \emph{molecules} is the inductive subclass of oriented graded posets closed under isomorphisms and generated by the following clauses.
\begin{enumerate}
    \item (\textit{Point}). The point is a molecule.
    \item (\textit{Paste}). Let $U$, $V$ be molecules, let $k < \min \set{ \dim U, \dim V }$, and let $\varphi\colon \bound{k}{+}U \incliso \bound{k}{-}V$ be an isomorphism.
	    Then $U \cpiso{k}{\varphi} V$ is a molecule.
    \item (\textit{Atom}). Let $U$, $V$ be round molecules of the same finite dimension and let $\varphi\colon \bound{}{}U \incliso \bound{}{}V$ be an isomorphism restricting to $\varphi^\alpha\colon \bound{}{\alpha} U \incliso \bound{}{\alpha} V$ for each $\alpha \in \set{ +, - }$.
	    Then $U \celto^\varphi V$ is a molecule.
\end{enumerate}
An \emph{atom} is a molecule with a greatest element.
Equivalently, it is a molecule whose final generating clause is either (\textit{Point}) or (\textit{Atom}).
\end{dfn}

\noindent
The following summarises some basic properties of molecules.

\begin{prop} \label{prop:molecules_basic_properties}
	Let $U$ be a molecule, $n \in \mathbb{N}$, $\alpha \in \set{+, -}$, $x \in U$.
	Then
	\begin{enumerate}
		\item $U$ is globular,
		\item $\bound{n}{\alpha}U$ is a molecule,
		\item $\clset{x}$ is an atom and is round.
	\end{enumerate}
\end{prop}

\begin{prop} \label{prop:molecules_unique_iso}
	Let $U$, $V$ be molecules, $k \in \mathbb{N}$.
	Then
	\begin{enumerate}
		\item if $U$ and $V$ are isomorphic, there exists a unique isomorphism $\varphi\colon U \iso V$,
		\item if $U \cpiso{k}{\varphi} V$ or $U \celto^\varphi V$ are defined, they are defined for a unique $\varphi$.
	\end{enumerate}
\end{prop}

\begin{rmk}
	This allows us to write $U \cp{k} V$ and $U \celto V$, omitting the specific isomorphism, when $U$ and $V$ are molecules and the constructions are defined.
	It will also allow us to be relaxed about the distinction between isomorphism and equality of molecules.
\end{rmk}

\begin{dfn}[Regular directed complex]
A \emph{regular directed complex} is an oriented graded poset $P$ with the property that, for all $x \in P$, the closed subset $\clset{x}$ is an atom.
We write $\rdcpx$ for the full subcategory of $\ogpos$ whose objects are regular directed complexes.
\end{dfn}

\begin{comm}
	While superficially different, this is equivalent to the definition of regular directed complex given in \cite{hadzihasanovic2020diagrammatic}.
	On the other hand, our definition of \emph{molecule} here corresponds to a \emph{molecule in a regular directed complex}, or \emph{regular molecule}, and is more restrictive than the definition in \cite{steiner1993algebra}.
	Note that by Proposition \ref{prop:molecules_basic_properties}, every molecule is a regular directed complex.
\end{comm}

\begin{dfn}[Positive least element]
	Let $P$ be an oriented graded poset, $\bot \in P$.
	We say that $\bot$ is a \emph{positive least element of $P$} if $\bot$ is the least element of $P$ and $\cofaces{}{}\bot = \cofaces{}{+}\bot$.
	We let $\ogposbot$ denote the full subcategory of $\ogpos$ on oriented graded posets with a positive least element.
\end{dfn}

\noindent
Freely adjoining a positive least element and, respectively, deleting the least element exhibit an equivalence between $\ogpos$ and $\ogposbot$.

\begin{prop} \label{prop:augmentation}
	There exists a pair of functors
\[
		\augm{(-)}\colon \ogpos \to \ogposbot, \quad \quad
		\dimin{(-)}\colon \ogposbot \to \ogpos
\]
	inverse to each other up to natural isomorphism.
\end{prop}

\noindent 
The following is a useful property of regular directed complexes.

\begin{dfn}[Oriented thin graded poset]
	Let $P$ be an oriented graded poset with a positive least element.
	We say that $P$ is \emph{oriented thin} if, for all $x, y \in P$ such that $x \leq y$ and $\codim{x}{y} = 2$, the interval $[x, y]$ is of the form
\[\begin{tikzcd}[sep=small]
	& y \\
	{z_1} && {z_2} \\
	& x
	\arrow[no head, "{\alpha}"', from=1-2, to=2-1]
	\arrow[no head, "{\beta}", from=1-2, to=2-3]
	\arrow[no head, "{\gamma}"', from=2-1, to=3-2]
	\arrow[no head, "{-\alpha\beta\gamma}", from=2-3, to=3-2]
\end{tikzcd}\]
for exactly two elements $z_1, z_2$, and for some $\alpha, \beta, \gamma \in \set{+, -}$.
\end{dfn}

\begin{prop} \label{prop:oriented_diamond_rdc}
	Let $P$ be a regular directed complex.
	Then $\augm{P}$ is an oriented thin graded poset.
\end{prop}

\noindent
The connection between oriented graded posets and strict $\omega$\nbd categories is given by the fact that (isomorphism classes of) molecules form a strict $\omega$\nbd category with pasting at the $k$\nbd boundary as $k$\nbd composition.
The fibred version of this result implies that (isomorphism classes of) molecules \emph{over an oriented graded poset $P$} form a strict $\omega$\nbd category.

In what follows, we recall the \emph{single-set} definition of strict $\omega$\nbd category, which is most natural in this context, and state these results more precisely.

\begin{dfn}[Reflexive $\omega$-graph] \index{reflexive $\omega$-graph} \index{cell}
	A \emph{reflexive $\omega$\nbd graph} is a set $X$, whose elements are called \emph{cells}, together with, for all $n \in \mathbb{N}$, operators $\bound{n}{-}, \bound{n}{+}\colon X \to X$
	called \emph{input} and \emph{output $n$\nbd boundary}, satisfying the following axioms.
	\begin{enumerate}
		\item (\textit{Finite dimension}). For all $t \in X$, there exists $n \in \mathbb{N}$ such that
			\[
				\bound{n}{-}t = \bound{n}{+}t = t.
			\]
		\item (\textit{Globularity}). For all $t \in X$, $k, n \in \mathbb{N}$, and $\alpha, \beta \in \set{ +, - }$,
			\[
				\bound{k}{\alpha}(\bound{n}{\beta}t) =
				\begin{cases}
					\bound{k}{\alpha}t & \text{if $k < n$,} \\
					\bound{n}{\beta}t & \text{if $k \geq n$.}
				\end{cases}
			\]
	\end{enumerate}
	If $t$ is a cell in a reflexive $\omega$\nbd graph, the \emph{dimension of $t$} is the natural number $\dim{t} \eqdef \min \set{ n \in \mathbb{N} \mid \bound{n}{-}t = \bound{n}{+}t = t }$.
\end{dfn}

\begin{dfn}[Composable pair of cells]
	Let $t, u$ be a pair of cells in a reflexive $\omega$\nbd graph, $k \in \mathbb{N}$.
	We say that $t$ and $u$ are \emph{$k$\nbd composable} if $\bound{k}{+}t = \bound{k}{-}u$.
	We write
	\[
		X \cpable{k} X \eqdef \set{ (t, u) \in X \times X \mid \bound{k}{+}t = \bound{k}{-}u }.
	\]
	for the set of $k$\nbd composable pairs of cells in $X$.
\end{dfn}

\begin{dfn}[Strict $\omega$-category]
	A \emph{strict $\omega$\nbd category} is a reflexive $\omega$\nbd graph $X$ together with, for all $k \in \mathbb{N}$, an operation $- \cp{k} -\colon X \cpable{k} X \to X$
	called \emph{$k$\nbd composition}, satisfying the following axioms.
	\begin{enumerate}
		\item (\textit{Compatibility with boundaries}).
			For all $k$\nbd composable pairs of cells $t, u$, all $n \in \mathbb{N}$, and $\alpha \in \set{+, -}$,
			\[
				\bound{n}{\alpha}(t \cp{k} u) =
				\begin{cases}
					\bound{n}{\alpha}t = \bound{n}{\alpha}u & \text{if $n < k$},\\
					\bound{k}{-}t & \text{if $n = k$, $\alpha = -$}, \\
					\bound{k}{+}u & \text{if $n = k$, $\alpha = +$}, \\
					\bound{n}{\alpha}t \cp{k} \bound{n}{\alpha}u & \text{if $n > k$}.
				\end{cases}
			\]
		\item (\textit{Associativity}).
			For all cells $t, u, v$ such that either side of the equation is defined, $(t \cp{k} u) \cp{k} v = t \cp{k} (u \cp{k} v)$.
		\item (\textit{Unitality}).
			For all cells $t$, $t \cp{k} \bound{k}{+}t = \bound{k}{-}t \cp{k} t = t$.
		\item (\textit{Interchange}).
			For all cells $t, t', u, u'$ and $n > k$ such that the left-hand side is defined, $(t \cp{n} t') \cp{k} (u \cp{n} u') = (t \cp{k} u) \cp{n} (t' \cp{k} u')$.
	\end{enumerate}
	Given a strict $\omega$\nbd category $X$ and $n \in \mathbb{N}$, we let $\skel{n}{X}$ denote its \emph{$n$-skeleton}, that is, its restriction to cells of dimension $\leq n$.
	A strict $\omega$\nbd category is a \emph{strict $n$\nbd category} if it is equal to its $n$\nbd skeleton.
\end{dfn}

\begin{dfn}[Strict functor of strict $\omega$-categories]
Let $X, Y$ be strict $\omega$-categories.
A \emph{strict functor} $f\colon X \to Y$ is a function such that, for all $k, n \in \mathbb{N}$, $\alpha \in \set{+, -}$, and $k$\nbd composable cells $t, u$ in $X$,
\[
	f(\bound{n}{\alpha}t) = \bound{n}{\alpha}f(t), \quad \quad f(t \cp{k} u) = f(t) \cp{k} f(u).
\]
Strict $\omega$\nbd categories and strict functors form a category $\omegacat$.
\end{dfn}

\begin{dfn}[Generating sets and bases]
	Let $X$ be a strict $\omega$\nbd category and $\gener{S}$ a set of cells in $X$.
	The set $\cspan{\gener{S}}$ is the smallest set such that
	\begin{enumerate}
		\item if $t \in \gener{S}$, then $t \in \cspan{\gener{S}}$,
		\item for all $k \in \mathbb{N}$, if $t, u \in \cspan{\gener{S}}$ are $k$\nbd composable, then $t \cp{k} u \in \cspan{\gener{S}}$.
	\end{enumerate}
	A \emph{generating set for $X$} is a set $\gener{S}$ of cells such that $\cspan{\gener{S}}$ contains every cell in $X$.
	A \emph{basis for $X$} is a minimal generating set.
\end{dfn}

\begin{lem} \label{lem:functors_equal_on_generating_set}
	Let $f,g\colon X \to Y$ be strict functors and let $\gener{S}$ be a generating set for $X$.
	If $f(t) = g(t)$ for all $t \in \gener{S}$, then $f = g$.
\end{lem}

\begin{dfn}[Isomorphism classes of molecules]
	For each oriented graded poset $P$, let $\isocl{P}$ denote its isomorphism class in $\ogpos$.
	We let
	\begin{align*}
		\molec & \eqdef \set{ \isocl{U} \mid \text{$U$ is a molecule} }, \\
		\atom & \eqdef \set{ \isocl{U} \mid \text{$U$ is an atom} } \subset \molec.
	\end{align*}
\end{dfn}

\begin{prop} \label{prop:molecules_form_omegacat}
	For all $n, k \in \mathbb{N}$ and $\alpha \in \set{+, -}$, let
	\begin{align*}
		\bound{n}{\alpha}&\colon \molec \to \molec, &
				 & \isocl{U} \mapsto \isocl{\bound{n}{\alpha}U}, \\
		- \cp{k} -&\colon \molec \cpable{k} \molec \to \molec, &
			  	& \isocl{U}, \isocl{V} \mapsto \isocl{U \cp{k} V}.
	\end{align*}
	Then $\molec$ together with these operations is a strict $\omega$\nbd category.
	Moreover,
	\begin{enumerate}
		\item for all molecules $U$, $\dim{\isocl{U}} = \dim{U}$,
		\item $\atom$ is a basis for $\molec$.
	\end{enumerate}
\end{prop}

\begin{dfn}[Molecules over an oriented graded poset] \index{$\molecin{P}$} \index{$\atomin{P}$}
	For each morphism $f\colon U \to P$ of oriented graded posets, let $\isocl{f}$ denote its isomorphism class in the slice category $\slice{\ogpos}{P}$.
	Given an oriented graded poset $P$, we let
	\begin{align*}
		\molecin{P} & \eqdef \set{ \isocl{f\colon U \to P} \mid \text{$U$ is a molecule} }, \\
		\atomin{P} & \eqdef \set{ \isocl{f\colon U \to P} \mid \text{$U$ is an atom} } \subseteq \molecin{P},
	\end{align*}
	which we call \emph{molecules} and \emph{atoms over $P$}.
	For all $k \in \mathbb{N}$ and $\alpha \in \set{+, -}$,
	\[
		\bound{k}{\alpha}\colon \molecin{P} \to \molecin{P}, \quad \quad
		\isocl{f\colon U \to P} \mapsto \isocl{\restr{f}{\bound{k}{\alpha}U}\colon \bound{k}{\alpha}U \to P}
	\]
	make $\molecin{P}$ a reflexive $\omega$\nbd graph.
	If $[f\colon U \to P]$, $[g\colon V \to P]$ are $k$\nbd composable molecules over $P$, then there exists a unique isomorphism $\varphi$ such that
	\[
	\begin{tikzcd}
		\bound{k}{+}U & \bound{k}{-}V & V \\
		U && P
		\arrow["\varphi", hook, from=1-1, to=1-2]
		\arrow[hook, from=1-2, to=1-3]
		\arrow[hook', from=1-1, to=2-1]
		\arrow["f", from=2-1, to=2-3]
		\arrow["g", from=1-3, to=2-3]
		\end{tikzcd}
	\]
	commutes, which induces, by the universal property of $U \cp{k} V$, a unique morphism $f \cp{k} g$ such that the following diagram commutes:
\[\begin{tikzcd}
	\bound{k}{+}U & \bound{k}{-}V & V \\
	U && U \cp{k} V \\
        &&& P.
	\arrow["\varphi", hook, from=1-1, to=1-2]
	\arrow[hook, from=1-2, to=1-3]
	\arrow[hook', from=1-1, to=2-1]
	\arrow["\imath_U", hook, from=2-1, to=2-3]
	\arrow["\imath_V", hook', from=1-3, to=2-3]
	\arrow["\lrcorner"{anchor=center, pos=0.125, rotate=180}, draw=none, from=2-3, to=1-1]
        \arrow["f", bend right, looseness=.5, from=2-1, to=3-4]
        \arrow["g", bend left, from=1-3, to=3-4]
        \arrow["f \cp{k} g"', dotted, from=2-3, to=3-4]
\end{tikzcd}\]
\end{dfn}

\begin{prop}
	Let $P$ be an oriented graded poset and, for each $k \in \mathbb{N}$,
	\begin{align*}
		- \cp{k} -\colon \molecin{P} \cpable{k} \molecin{P} & \to \molecin{P}, \\
		\isocl{f\colon U \to P}, \isocl{g\colon V \to P} & \mapsto \isocl{f \cp{k} g\colon U \cp{k} V \to P}.
	\end{align*}
	Then $\molecin{P}$ together with these composition operations is a strict $\omega$\nbd category, which has the set $\atomin{P}$ as a basis.
	In particular, if $\dim{P} \leq n$, then $\molecin{P}$ is a strict $n$\nbd category.
	This assignment extends to a functor $\molecin{-}\colon \ogpos \to \omegacat$.
\end{prop}

\noindent When $P$ is a regular directed complex, $\molecin{P}$ admits a basis whose elements are in bijection with the elements of $P$, as a consequence of the following result.
This comes from very strong rigidity properties of atoms, which do not generalise to other oriented graded posets.

\begin{dfn}[Local embedding of oriented graded posets]
	A morphism $f\colon P \to Q$ of oriented graded posets is a \emph{local embedding} if, for all $x \in P$, the restriction $\restr{f}{\clset{x}}$ is an inclusion, hence determines an isomorphism between $\clset{x}$ and its image $\clset{f(x)}$.
\end{dfn}

\begin{prop} \label{prop:rdcpx_local_embeddings}
	Let $f\colon P \to Q$ be a morphism of regular directed complexes.
	Then $f$ is a local embedding.
\end{prop}

\begin{cor} \label{cor:basis_of_omegacat_presented_by_rdcpx}
	If $P$ is a regular directed complex, $\set{ \isocl{ \clset{x} \incl P } \mid x \in P }$
	is a basis for the $\omega$\nbd category $\molecin{P}$.
\end{cor}

\begin{dfn}[Diagram in a strict $\omega$-category]
	Let $X$ be a strict $\omega$\nbd category and $P$ a regular directed complex.
	A \emph{diagram of shape $P$ in $X$} is a strict functor $d\colon \molecin{P} \to X$.
	A diagram is a \emph{pasting diagram} if its shape is a molecule.
\end{dfn}

\section{Layerings, flow graphs, and orderings} \label{sec:layerings}

\noindent
Since pasting of molecules satisfies the axioms of strict $\omega$\nbd categories, it is clear that every molecule admits multiple pasting decompositions.
However, the space of possible decompositions can at least in part be constrained by considering decompositions of a special type, called \emph{layerings}, where each factor (layer) contains exactly one maximal cell of dimension higher than the pasting dimension.
Such decompositions have played a role in most past approaches to higher-categorical diagrams --- see for example \cite{forest2022unifying} --- but we first studied them systematically in 
\cite{hadzihasanovic2023higher}.
We give an overview of the main notions and results, and refer to 
\cite[Chapter 4]{hadzihasanovic2024combinatorics} for proofs.

\begin{dfn}[Layering]
Let $U$ be a molecule, $-1 \leq k < \dim{U}$.
A \emph{$k$\nbd layering of $U$} is a sequence $(\order{i}{U})_{i=1}^m$ of molecules such that $U = \order{1}{U} \cp{k} \ldots \cp{k} \order{m}{U}$
and $\size{\bigcup_{i > k} \grade{i}{(\maxel{\order{j}{U}})}} = 1$ for all $j \in \set{1, \ldots, m}$, that is, each ``layer'' $\order{j}{U}$ contains a single maximal element of dimension $> k$.
\end{dfn}

\begin{exm}
If $U$ is the molecule encoding the 2\nbd dimensional pasting diagram shape
\[	\begin{tikzcd}
	\bullet & \bullet & \bullet & \bullet
	\arrow[""{name=0, anchor=center, inner sep=0}, "0"', curve={height=12pt}, from=1-1, to=1-2]
	\arrow[""{name=1, anchor=center, inner sep=0}, "3", curve={height=-12pt}, from=1-1, to=1-2]
	\arrow["1", from=1-2, to=1-3]
	\arrow[""{name=2, anchor=center, inner sep=0}, "4", curve={height=-12pt}, from=1-3, to=1-4]
	\arrow[""{name=3, anchor=center, inner sep=0}, "2"', curve={height=12pt}, from=1-3, to=1-4]
	\arrow["1"', shorten <=3pt, shorten >=3pt, Rightarrow, from=3, to=2]
	\arrow["0", shorten <=3pt, shorten >=3pt, Rightarrow, from=0, to=1]
\end{tikzcd}\]
then $U$ admits a single 0\nbd layering with layers
\[\begin{tikzcd}
	\bullet & \bullet & \bullet & \bullet & \bullet & \bullet
	\arrow[""{name=0, anchor=center, inner sep=0}, "0"', curve={height=12pt}, from=1-1, to=1-2]
	\arrow[""{name=1, anchor=center, inner sep=0}, "3", curve={height=-12pt}, from=1-1, to=1-2]
	\arrow[""{name=2, anchor=center, inner sep=0}, "4", curve={height=-12pt}, from=1-5, to=1-6]
	\arrow[""{name=3, anchor=center, inner sep=0}, "2"', curve={height=12pt}, from=1-5, to=1-6]
	\arrow["1", from=1-3, to=1-4]
	\arrow["1"', shorten <=3pt, shorten >=3pt, Rightarrow, from=3, to=2]
	\arrow["0", shorten <=3pt, shorten >=3pt, Rightarrow, from=0, to=1]
\end{tikzcd}\]
	and two 1\nbd layerings with layers
\[\begin{tikzcd}
	\bullet & \bullet & \bullet & \bullet & \bullet & \bullet & \bullet & \bullet
	\arrow[""{name=0, anchor=center, inner sep=0}, "0"', curve={height=12pt}, from=1-1, to=1-2]
	\arrow[""{name=1, anchor=center, inner sep=0}, "3", curve={height=-12pt}, from=1-1, to=1-2]
	\arrow["2", from=1-3, to=1-4]
	\arrow["1", from=1-2, to=1-3]
	\arrow["3", from=1-5, to=1-6]
	\arrow["1", from=1-6, to=1-7]
	\arrow[""{name=2, anchor=center, inner sep=0}, "4", curve={height=-12pt}, from=1-7, to=1-8]
	\arrow[""{name=3, anchor=center, inner sep=0}, "2"', curve={height=12pt}, from=1-7, to=1-8]
	\arrow["0", shorten <=3pt, shorten >=3pt, Rightarrow, from=0, to=1]
	\arrow["1", shorten <=3pt, shorten >=3pt, Rightarrow, from=3, to=2]
\end{tikzcd}\;,\]
\[
	\begin{tikzcd}
	\bullet & \bullet & \bullet & \bullet & \bullet & \bullet & \bullet & \bullet
	\arrow[""{name=0, anchor=center, inner sep=0}, "0"', curve={height=12pt}, from=1-5, to=1-6]
	\arrow[""{name=1, anchor=center, inner sep=0}, "3", curve={height=-12pt}, from=1-5, to=1-6]
	\arrow["4", from=1-7, to=1-8]
	\arrow["1", from=1-6, to=1-7]
	\arrow["0", from=1-1, to=1-2]
	\arrow["1", from=1-2, to=1-3]
	\arrow[""{name=2, anchor=center, inner sep=0}, "4", curve={height=-12pt}, from=1-3, to=1-4]
	\arrow[""{name=3, anchor=center, inner sep=0}, "2"', curve={height=12pt}, from=1-3, to=1-4]
	\arrow["0", shorten <=3pt, shorten >=3pt, Rightarrow, from=0, to=1]
	\arrow["1", shorten <=3pt, shorten >=3pt, Rightarrow, from=3, to=2]
\end{tikzcd}\] 
respectively.
\end{exm}

\begin{dfn}[Frame and layering dimension]
Let $U$ be a molecule.
The \emph{frame dimension} of $U$ is the integer
\begin{equation*}
    \frdim{U} \eqdef \dim{\bigcup \set{(\clset{{x}} \cap \clset{{y}}) \mid
    x, y \in \maxel{U}, x \neq y } }.
\end{equation*}
The \emph{layering dimension} of $U$ is the integer
\begin{equation*}
    \lydim{U} \eqdef \min \set{ k \geq -1 \mid \size{\bigcup_{i > k+1} \grade{i}{(\maxel{U})}} \leq 1 }.
\end{equation*}
\end{dfn}

\begin{prop} \label{prop:layerings_exist}
	Let $U$ be a molecule.
	Then
	\begin{enumerate}
		\item there exists $k \geq -1$ such that $U$ admits a $k$\nbd layering,
		\item if $U$ admits a $k$\nbd layering, it admits an $\ell$\nbd layering for all $k \leq \ell < \dim{U}$,
		\item $\frdim{U} \leq \min \set{k \geq -1 \mid \text{$U$ admits a $k$\nbd layering}} \leq \lydim{U}$.
	\end{enumerate}
\end{prop}

\begin{lem} \label{lem:induction_on_layering_dimension}
Let $U$ be a molecule.
Then
\begin{enumerate}
	\item $\lydim{U} = -1$ if and only if $\frdim{U} = -1$ if and only if $U$ is an atom,
	\item if $k \geq 0$ and $(\order{i}{U})_{i=1}^m$ is a $k$\nbd layering of $U$, then for each $i \in \set{1, \ldots, m}$, $\lydim{\order{i}{U}} < k$.
\end{enumerate}
\end{lem}

\begin{comm} \label{comm:induction_on_lydim}
Proposition \ref{prop:layerings_exist} in conjunction with Lemma \ref{lem:induction_on_layering_dimension} allows us to prove properties of molecules \emph{by induction on their layering dimension}.
That is, to prove that a property holds of all molecules $U$, it suffices to
\begin{itemize}
    \item prove that it holds when $\lydim{U} = -1$, that is, when $U$ is an atom,
    \item prove that it holds when $k \eqdef \lydim{U} \geq 0$, assuming that it holds of all the $(\order{i}{U})_{i=1}^m$ in a $k$\nbd layering of $U$.
\end{itemize}
\end{comm}

\begin{dfn}[Flow graph]
Let $P$ be an oriented graded poset, $k \geq -1$.
The \emph{$k$\nbd flow graph of $P$} is the directed graph $\flow{k}{P}$ whose
\begin{itemize}
    \item set of vertices is $\bigcup_{i > k} \grade{i}{P}$, and
    \item set of edges is
	    \[
		    \set{(x, y) \mid \faces{k}{+}x \cap \faces{k}{-}y \neq \varnothing},
		\]
	where the source of $(x, y)$ is $x$ and the target is $y$. 
\end{itemize}
\end{dfn}

\begin{dfn}[Maximal flow graph]
Let $P$ be a finite-dimensional oriented graded poset, $k \geq -1$.
The \emph{maximal $k$\nbd flow graph of $P$} is the induced subgraph $\maxflow{k}{P}$ of $\flow{k}{P}$ on the vertex set
\begin{equation*}
    \bigcup_{i > k} \grade{i}{(\maxel{P})} \subseteq \bigcup_{i > k} \grade{i}{P}.
\end{equation*}
\end{dfn}

\begin{exm}
If $U$ is the molecule encoding the 2\nbd dimensional pasting diagram shape
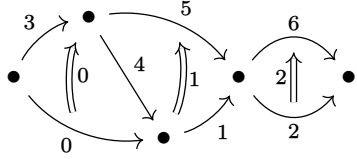
\[
		\begin{tikzcd}[sep=small]
	& \bullet \\
	\bullet &&& \bullet && \bullet \\
	&& \bullet
	\arrow["3", curve={height=-6pt}, from=2-1, to=1-2]
	\arrow[""{name=0, anchor=center, inner sep=0}, "5", curve={height=-12pt}, from=1-2, to=2-4]
	\arrow[""{name=1, anchor=center, inner sep=0}, "6", curve={height=-18pt}, from=2-4, to=2-6]
	\arrow[""{name=2, anchor=center, inner sep=0}, "2"', curve={height=18pt}, from=2-4, to=2-6]
	\arrow[""{name=3, anchor=center, inner sep=0}, "0"', curve={height=12pt}, from=2-1, to=3-3]
	\arrow["1"', curve={height=6pt}, from=3-3, to=2-4]
	\arrow["4", from=1-2, to=3-3]
	\arrow["0"', curve={height=-6pt}, shorten <=7pt, shorten >=3pt, Rightarrow, from=3, to=1-2]
	\arrow["2", shorten <=5pt, shorten >=5pt, Rightarrow, from=2, to=1]
	\arrow["1"', curve={height=6pt}, shorten <=3pt, shorten >=7pt, Rightarrow, from=3-3, to=0]
\end{tikzcd}\]
the 0\nbd flow graph $\flow{0}{U}$ is
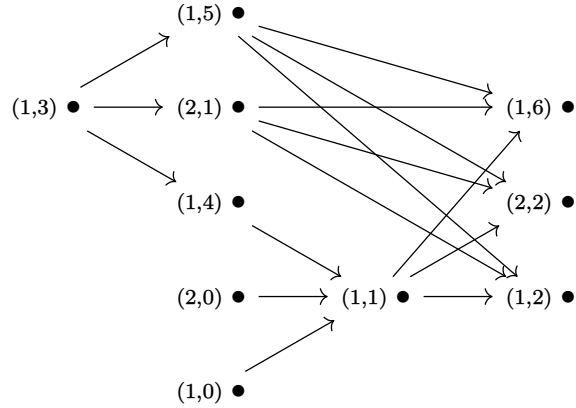
\[\begin{tikzcd}
	& {{\scriptstyle (1, 5)}\;\bullet} \\
	{{\scriptstyle (1, 3)}\;\bullet} & {{\scriptstyle (2, 1)}\;\bullet} && {{\scriptstyle (1, 6)}\;\bullet} \\
	& {{\scriptstyle (1, 4)}\;\bullet} && {{\scriptstyle (2, 2)}\;\bullet} \\
	& {{\scriptstyle (2, 0)}\;\bullet} & {{\scriptstyle (1, 1)}\;\bullet} & {{\scriptstyle (1, 2)}\;\bullet} \\
	& {{\scriptstyle (1, 0)}\;\bullet}
	\arrow[from=4-2, to=4-3]
	\arrow[from=5-2, to=4-3]
	\arrow[from=3-2, to=4-3]
	\arrow[from=2-1, to=3-2]
	\arrow[from=2-1, to=1-2]
	\arrow[from=1-2, to=3-4]
	\arrow[from=4-3, to=3-4]
	\arrow[from=2-1, to=2-2]
	\arrow[from=1-2, to=2-4]
	\arrow[from=1-2, to=4-4]
	\arrow[from=4-3, to=4-4]
	\arrow[from=4-3, to=2-4]
	\arrow[from=2-2, to=2-4]
	\arrow[from=2-2, to=3-4]
	\arrow[from=2-2, to=4-4]
\end{tikzcd}\]
	and the maximal 0\nbd flow graph $\maxflow{0}{U}$ is its induced subgraph
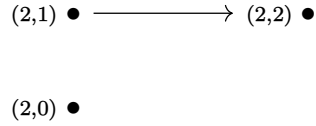
\[\begin{tikzcd}
	{{\scriptstyle (2, 1)}\;\bullet} && {{\scriptstyle (2, 2)}\;\bullet} \\
	{{\scriptstyle (2, 0)}\;\bullet}
	\arrow[from=1-1, to=1-3]
\end{tikzcd}\]
	while the 1\nbd flow graph $\flow{1}{U}$ is
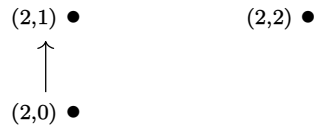
\[\begin{tikzcd}
	{{\scriptstyle (2, 1)}\;\bullet} && {{\scriptstyle (2, 2)}\;\bullet} \\
	{{\scriptstyle (2, 0)}\;\bullet}
	\arrow[from=2-1, to=1-1]
\end{tikzcd}\]
	and it is equal to $\maxflow{1}{U}$, since every 2\nbd dimensional element of $U$ is maximal.
\end{exm}

\begin{dfn}[Ordering of a molecule]
Let $U$ be a molecule, $k \geq -1$.
A \emph{$k$\nbd ordering of $U$} is a topological sort of $\maxflow{k}{U}$.
\end{dfn}

\begin{rmk}
	A $k$\nbd ordering of $U$ exists if and only if $\maxflow{k}{U}$ is acyclic.
\end{rmk}

\begin{prop} \label{prop:layerings_induce_orderings}
Let $U$ be a molecule, $k \geq -1$, and let
\begin{align*}
    \layerings{k}{U} & \eqdef \set{\text{$k$\nbd layerings $(\order{i}{U})_{i=1}^m$ of $U$ up to layer-wise isomorphism}}, \\
    \orderings{k}{U} & \eqdef \set{\text{$k$\nbd orderings $(\order{i}{x})_{i=1}^m$ of $U$}}.
\end{align*}
For each $k$\nbd layering $(\order{i}{U})_{i=1}^m$ of $U$ and each $i \in \set{1, \ldots, m}$, let $\order{i}{x}$ be the only element of $\bigcup_{j > k}\grade{j}{(\maxel{U})}$ in the layer $\order{i}{U}$.
Then the assignment
\[    
	\lto{k}{U}\colon (\order{i}{U})_{i=1}^m \mapsto (\order{i}{x})_{i=1}^m
\]
determines an injective function $\layerings{k}{U} \incl \orderings{k}{U}$.
Moreover, if $U$ admits a $k$\nbd layering, then for all $\ell > k$, the function $\lto{k}{U}$ is a bijection.
\end{prop}

\section{Frame-acyclicity and polygraphs} \label{sec:frameacy}

\noindent
In this section, we study what seems to be the mildest acyclicity condition on an oriented graded poset $P$ guaranteeing that $\molecin{P}$ is freely generated in the sense of polygraphs.
On molecules, this condition, which we call \emph{frame-acyclicity}, is (non-trivially) equivalent to what is called being \emph{split} in 
\cite{steiner1993algebra}; our treatment elucidates its status as an acyclicity condition, which was not originally recognised.

Frame-acyclicity was first defined in \cite{hadzihasanovic2021smash}, and in 
\cite{hadzihasanovic2023higher} we studied its role in algorithmic properties of higher-dimensional rewriting.
However, self-contained proofs of the main results related to frame-acyclicity have not appeared in print; this section is meant to fix this gap.

\begin{dfn}[Submolecules]
Let $V \subseteq U$ be molecules.
We say that $V$ is a \emph{submolecule} of $U$, and write $V \submol U$, if $V$ is a factor in a pasting decomposition of $U$.
\end{dfn}

\begin{rmk}
Equivalently, $\submol$ can be characterised as the smallest partial order on molecules such that $U, V \submol U \cp{k} V$ whenever the latter is defined, once $U$ and $V$ are identified with their images in the pasting.
\end{rmk}

\begin{dfn}[Frame-acyclic molecule] \index{molecule!acyclic!frame-}
Let $U$ be a molecule.
We say that $U$ is \emph{frame-acyclic} if for all submolecules $V \submol U$, if $r \eqdef \frdim{V}$, then $\maxflow{r}{V}$ is acyclic.
\end{dfn}

\begin{dfn}[Oriented graded poset with frame-acyclic molecules] \index{oriented graded poset!with frame-acyclic molecules}
	Let $P$ be an oriented graded poset.
	We say that $P$ \emph{has frame-acyclic molecules} if, for all molecules $U$, if there exists a morphism $f\colon U \to P$, then $U$ is frame-acyclic.
\end{dfn}

\noindent
We recall without proof the following result \cite[Theorem 121]{hadzihasanovic2023higher}, which implies that this condition is only non-trivial starting from dimension 4.

\begin{thm} \label{thm:dim3_frame_acyclic}
	Let $U$ be a molecule, $\dim{U} \leq 3$.
	Then $U$ is frame-acyclic.
\end{thm}

\begin{cor}
	Let $P$ be an oriented graded poset, $\dim{P} \leq 3$.
	Then $P$ has frame-acyclic molecules.
\end{cor}

\begin{exm}
	The dimensional bound in Theorem \ref{thm:dim3_frame_acyclic} is strict: \cite[Example 126]{hadzihasanovic2023higher}, exhibits a 4\nbd dimensional molecule which is not frame-acyclic.
\end{exm}

\begin{lem} \label{lem:frdim_layerings_bijection_with_orderings}
Let $U$ be a molecule.
Suppose that for all submolecules $V \submol U$, if $r \eqdef \frdim{V}$, then $V$ admits an $r$\nbd layering.
Then for all $k \geq \frdim{U}$ the function $\lto{k}{U}\colon \layerings{k}{U} \incl \orderings{k}{U}$ is a bijection.
\end{lem}
\begin{proof}
Let $r \eqdef \frdim{U}$.
By assumption, there exists an $r$\nbd layering of $U$, so by Proposition \ref{prop:layerings_induce_orderings}	it suffices to show that $\lto{r}{U}$ is a bijection.

Given two $r$\nbd orderings $(\order{i}{x})_{i=1}^m$ and $(\order{i}{y})_{i=1}^m$, there exists a unique permutation $\sigma$ such that $\order{i}{x} = \order{\sigma(i)}{y}$ for all $i \in \set{1, \ldots, m}$.
Let $\fun{d}((\order{i}{x})_{i=1}^m, (\order{i}{y})_{i=1}^m)$ be the number of pairs $(j, j')$ such that $j < j'$ but $\sigma(j') < \sigma(j)$.
Under the assumption that $(\order{i}{x})_{i=1}^m$ is in the image of $\lto{r}{U}$, we will prove that $(\order{i}{y})_{i=1}^m$ is also in the image of $\lto{r}{U}$ by induction on $\fun{d}((\order{i}{x})_{i=1}^m, (\order{i}{y})_{i=1}^m)$.
Since the image of $\lto{r}{U}$ is not empty, this will suffice to prove that $\lto{r}{U}$ is surjective, hence bijective by Proposition \ref{prop:layerings_induce_orderings}.

If $\fun{d}((\order{i}{x})_{i=1}^m, (\order{i}{y})_{i=1}^m) = 0$, then $\order{i}{x} = \order{i}{y}$ for all $i \in \set{1, \ldots, m}$, and there is nothing left to prove.

Suppose $\fun{d}((\order{i}{x})_{i=1}^m, (\order{i}{y})_{i=1}^m) > 0$.
Then there exists $j < m$ such that $\sigma(j+1) < \sigma(j)$.
Suppose $(\order{i}{x})_{i=1}^m$ is the image of the $r$\nbd layering $(\order{i}{U})_{i=1}^m$.
Let $V \submol U$ be the image of $\order{j}{U} \cp{r} \order{j+1}{U}$ in $U$, and let
\begin{equation*}
    z_1 \eqdef \order{j}{x} = \order{\sigma(j)}{y}, \quad z_2 \eqdef \order{j+1}{x} = \order{\sigma(j+1)}{y}.
\end{equation*}
Because $z_1$ comes before $z_2$ in one $r$\nbd ordering, but after in another, there can be no edge between them in $\maxflow{r}{U}$, so
\begin{equation*}
    \dim{ (\clset{{z_1}} \cap \clset{{z_2}}) } < r.
\end{equation*}
Since $z_1, z_2$ are the only maximal elements of dimension $> r$ in $V$, we deduce that $\ell \eqdef \frdim{V} < r$.
By assumption, there exists an $\ell$\nbd layering of $V$.
In particular, there exist molecules $\order{1}{V}, \order{2}{V}$ such that
\begin{enumerate}
    \item $z_i$ is in the image of $\order{i}{V}$ for all $i \in \set{1, 2}$, and
    \item $V$ is isomorphic to $\order{1}{V} \cp{\ell} \order{2}{V}$ or to $\order{2}{V} \cp{\ell} \order{1}{V}$.
\end{enumerate}
Without loss of generality suppose that $V$ is isomorphic to $\order{1}{V} \cp{\ell} \order{2}{V}$.
By the unitality and interchange properties of pasting, letting
\begin{align*}
    \order{j}{\tilde{U}} & \eqdef \bound{r}{-}\order{1}{V} \cp{\ell} \order{2}{V}, \\
    \order{j+1}{\tilde{U}} & \eqdef \order{1}{V} \cp{\ell} \bound{r}{+} \order{2}{V},
\end{align*}
we have that $V = \order{j}{\tilde{U}} \cp{r} \order{j+1}{\tilde{U}}$.
Letting $\order{i}{\tilde{U}} \eqdef \order{i}{U}$ for $i \notin \set{j, j+1}$, we have that $(\order{i}{\tilde{U}})_{i=1}^m$ is an $r$\nbd layering of $U$, and
\begin{equation*}
    \lto{r}{U}\colon (\order{i}{\tilde{U}})_{i=1}^m \mapsto (\order{i}{\tilde{x}})_{i=1}^m = (\order{1}{x}, \ldots, \order{j+1}{x}, \order{j}{x}, \ldots, \order{m}{x}).
\end{equation*}
Then $\fun{d}((\order{i}{\tilde{x}})_{i=1}^m, (\order{i}{y})_{i=1}^m) < \fun{d}((\order{i}{x})_{i=1}^m, (\order{i}{y})_{i=1}^m)$ and $(\order{i}{\tilde{x}})_{i=1}^m$ is in the image of $\lto{r}{U}$.
We conclude by the inductive hypothesis.
\end{proof}

\begin{comm} \label{comm:induction_on_submolecules} \index{submolecule!induction}
Let $\property{P}$ be a property of molecules such that, whenever $\mathrm{P}$ holds of a molecule $U$, then $\mathrm{P}$ holds of every submolecule $V \submol U$; the property of frame-acyclicity is of this sort.
Because every proper submolecule of $U$ has strictly fewer elements than $U$, the submolecule relation on $U$ is well-founded, and its minimal elements are the 0\nbd dimensional one-element subsets $\set{x} \submol U$ for each $x \in \grade{0}{U}$.

If we want to prove that $\property{P}$ implies $\property{Q}$ for all molecules, we can then proceed by \emph{induction on submolecules}: assume that a molecule $U$ satisfies $\property{P}$, then
\begin{itemize}
    \item prove that $\set{x}$ satisfies $\property{Q}$ for all $x \in \grade{0}{U}$,
    \item prove that $U$ satisfies $\property{Q}$ under the assumption that every proper submolecule $V \submol U$ satisfies $\property{Q}$.
\end{itemize}
\end{comm}

\begin{lem} \label{lem:intersection_of_maximal_elements}
Let $U$ be a molecule and let $x, y \in \maxel{U}$ such that $x \neq y$.
For all $k \geq \frdim{U}$, $\clset{{x}} \cap \clset{{y}} = (\bound{k}{-}x \cap \bound{k}{+}y) \cup (\bound{k}{+}x \cap \bound{k}{-}y)$.
\end{lem}
\begin{proof}
	See \cite[Proposition 6.4]{steiner1993algebra}.
\end{proof}

\begin{thm} \label{thm:frame_acyclic_has_frame_layerings}
Let $U$ be a molecule, $r \eqdef \frdim{U}$.
If $U$ is frame-acyclic, then $U$ admits an $r$\nbd layering.
\end{thm}
\begin{proof}
We proceed by induction on submolecules.
For all $x \in \grade{0}{U}$, we have $\frdim{\set{x}} = -1$, and $\set{x}$ admits the trivial $(-1)$\nbd layering, which proves the base case.

We construct a finite plane tree of submolecules $\order{j_1, \ldots, j_p}{U} \submol U$, as follows:
\begin{itemize}
    \item the root is $\order{}{U} \eqdef U$;
    \item if $\lydim{\order{j_1, \ldots, j_p}{U}} \leq r$, then we let $\lydim{\order{j_1, \ldots, j_p}{U}}$ be a leaf;
    \item if $k \eqdef \lydim{\order{j_1, \ldots, j_p}{U}} > r$, then we pick a $k$\nbd layering $(\order{i}{V})_{i=1}^q$ of $\order{j_1, \ldots, j_p}{U}$, and for each $i \in \set{1, \ldots, q}$, we let the image of $\order{i}{V}$ be a child $\order{j_1, \ldots, j_p, i}{U}$ of $\order{j_1, \ldots, j_p}{U}$.
\end{itemize}
By Lemma \ref{lem:induction_on_layering_dimension}, the layering dimension of the children of a node is strictly smaller than that of the node, so the procedure terminates.

Fix an $r$\nbd ordering $(\order{i}{x})_{i=1}^m$ of $U$; this is possible because $\maxflow{r}{U}$ is acyclic.
Let $V \eqdef \order{j_1, \ldots, j_p}{U}$ be a node of the tree.
We have
\begin{equation*}
    \bigcup_{j > r} \grade{j}{(\maxel{V})} =
    \sum_{i = 1}^m \bigcup_{j > r}
    \left(\grade{j}{(\maxel{V})} \cap \clset{{\order{i}{x}}}\right) \eqqcolon \sum_{i = 1}^m \order{i}{M};
\end{equation*}
the $\order{i}{M}$ form a partition because $\frdim{U} = r$, so every element of dimension $> r$ is in the closure of $\order{i}{x}$ for a unique $i \in \set{1, \ldots, m}$.
We claim that $V$ is isomorphic to $\order{1}{V} \cp{r} \ldots \cp{r} \order{m}{V}$
for some molecules $(\order{i}{V})_{i=1}^m$ such that, for each $i \in \set{1, \ldots, m}$, identifying $\order{i}{V}$ with its image in $V$, we have
\begin{equation*}
    \bigcup_{j > r} \grade{j}{(\maxel{\order{i}{V}})} =
    \order{i}{M}.
\end{equation*}
We will prove this by backward induction on the tree $\order{j_1, \ldots, j_p}{U}$.

Suppose $V$ is a leaf, so $\lydim{V} \leq r$.
Then $V$ admits an $r$\nbd layering.
For each $i \in \set{1, \ldots, m}$, fix a topological sort $(\order{i, j}{y})_{j=1}^{p_i}$ of the induced subgraph $\restr{\maxflow{r}{V}}{\order{i}{M}}$.
We claim that $((\order{i,j}{y})_{j=1}^{p_i})_{i=1}^m$ is an $r$\nbd ordering of $V$.

Suppose there is an edge from $x$ to $x'$ in $\maxflow{r}{V}$.
Then $x \in \order{i}{M}$, $x' \in \order{i'}{M}$ for a unique pair $i, i' \in \set{1, \ldots, m}$.
If $i = i'$, then $x = \order{i, j}{y}$ and $x' = \order{i, j'}{y}$ for some $j, j' \in \set{1, \ldots, p_i}$, and $j < j'$ because $(\order{i, j}{y})_{j=1}^{p_i}$ is a topological sort of $\restr{\maxflow{r}{V}}{\order{i}{M}}$.
If $i \neq i'$, then there exists
\begin{equation*}
    z \in \faces{r}{+}x \cap \faces{r}{-}x' \subseteq \clset{{\order{i}{x}}} \cap \clset{{\order{i'}{x}}}.
\end{equation*}
Since $\bound{r}{\alpha}\order{i}{x}$ and $\bound{r}{\alpha}\order{i'}{x}$ is pure and $r$\nbd dimensional for all $\alpha \in \set{+, -}$, by Lemma \ref{lem:intersection_of_maximal_elements}
\begin{equation*}
    z \in (\faces{r}{+}\order{i}{x} \cap \faces{r}{-}\order{i'}{x}) \cup (\faces{r}{-}\order{i}{x} \cap \faces{r}{+}\order{i'}{x}),
\end{equation*}
and by Lemma \ref{lem:faces_intersection} $\faces{r}{-}\order{i}{x} \cap \clset{{x}} \subseteq \faces{r}{-}x$ which is disjoint from $\faces{r}{+}x$, so $z \in \faces{r}{+}\order{i}{x} \cap \faces{r}{-}\order{i'}{x}$.
It follows that there is an edge from $\order{i}{x}$ to $\order{i'}{x}$ in $\maxflow{r}{U}$, so $i < i'$ because $(\order{i}{x})_{i=1}^m$ is a topological sort of $\maxflow{r}{U}$.
This proves that $((\order{i,j}{y})_{j=1}^{p_i})_{i=1}^m$ is an $r$\nbd ordering of $V$.

Let $W \submol V$, $\ell \eqdef \frdim{W}$.
If $V \neq U$ or $W \neq U$, then $W$ admits an $\ell$\nbd layering by the inductive hypothesis on proper submolecules of $U$.
If $W = V = U$ then $\ell = r$ and $W$ admits an $\ell$\nbd layering by Proposition \ref{prop:layerings_exist}.
In either case, $V$ satisfies the conditions of Lemma \ref{lem:frdim_layerings_bijection_with_orderings}, and since $r \geq \lydim{V} \geq \frdim{V}$, every $r$\nbd ordering of $V$ comes from an $r$\nbd layering of $V$.

It follows that $((\order{i,j}{y})_{j=1}^{p_i})_{i=1}^m$ comes from an $r$\nbd layering $((\order{i,j}{W})_{j=1}^{p_i})_{i=1}^m$, and we can define $\order{i}{V} \eqdef \order{i,1}{W} \cp{r} \ldots \cp{r} \order{i, p_i}{W}$ for each $i \in \set{1, \ldots, m}$, satisfying the desired condition.

Now, suppose that $V$ is not a leaf, so $k \eqdef \lydim{V} > r$, and $V$ has children $(\order{j}{W})_{j=1}^q$ forming a $k$\nbd layering of $V$.
By the inductive hypothesis, each of the $\order{j}{W}$ has a decomposition $\order{j, 1}{W} \cp{r} \ldots \cp{r} \order{j, m}{W}$
such that the maximal elements of dimension $> r$ in the image of $\order{j, i}{W}$ are contained in $\clset{{\order{i}{x}}}$.
Then, for each $i \in \set{1, \ldots, m}$ and $j, j' \in \set{1, \ldots, q}$, $\order{j, i}{W} \cap \order{j'}{W} \subseteq \order{j', i}{W}$,
so $\order{i}{V} \eqdef \order{1, i}{W} \cp{k} \ldots \cp{k} \order{q, i}{W}$ is defined.
Using interchange repeatedly, we conclude that $V$ is isomorphic to $\order{1}{V} \cp{r} \ldots \cp{r} \order{m}{V}$.

This concludes the induction on the tree $\order{j_1, \ldots, j_p}{U}$.
In particular, for the root $\order{}{U} = U$, the decomposition $\order{1}{U} \cp{r} \ldots \cp{r} \order{m}{U}$ satisfies
\begin{equation*}
    \bigcup_{j > r} \grade{j}{(\maxel{\order{i}{U}})} =
    \set{\order{i}{x}},
\end{equation*}
that is, $(\order{i}{U})_{i=1}^m$ is an $r$\nbd layering of $U$.
\end{proof}

\begin{cor} \label{cor:frame_acyclicity_equivalent_conditions}
Let $U$ be a molecule.
The following are equivalent:
\begin{enumerate}[label=(\alph*)]
    \item $U$ is frame-acyclic; \label{cond:frame_acyclic}
    \item for all $V \submol U$ and all $\frdim{V} \leq k < \dim{V}$, $V$ admits a $k$\nbd layering; \label{cond:frdim_layerings}
    \item for all $V \submol U$ and all $\frdim{V} \leq k < \dim{V}$, the sets $\layerings{k}{V}$ and $\orderings{k}{V}$ are non-empty and equinumerous. \label{cond:frdim_bijection}
\end{enumerate}
\end{cor}
\begin{proof}
	The implication from \ref{cond:frame_acyclic} to \ref{cond:frdim_layerings} is a consequence of Theorem \ref{thm:frame_acyclic_has_frame_layerings} together with Proposition \ref{prop:layerings_exist}.
The implication from \ref{cond:frdim_layerings} to \ref{cond:frdim_bijection} is Lemma \ref{lem:frdim_layerings_bijection_with_orderings}.
Finally, the implication from \ref{cond:frdim_bijection} to \ref{cond:frame_acyclic} follows from Proposition \ref{prop:layerings_induce_orderings}.
\end{proof}

\begin{rmk} \label{rmk:split_is_frameacy}
	Corollary \ref{cor:frame_acyclicity_equivalent_conditions} implies that a molecule is frame-acyclic if and only if it is \emph{split} in the sense of \cite{steiner1993algebra}.
	Thus this condition rephrases the splitness condition as an acyclicity condition, which will help us elucidate its connection to other, stronger acyclicity conditions.
\end{rmk}

\begin{dfn}[Cellular extension of a strict $\omega$-category]
	Let $X$ be a strict $\omega$\nbd category.
	A \emph{cellular extension of $X$} is a strict $\omega$\nbd category $X_\gener{S}$ together with a pushout diagram
\[
	\begin{tikzcd}
		{\coprod_{e \in \gener{S}} \molecin{\bound{}{}U_e}} &&& {\coprod_{e \in \gener{S}} \molecin{U_e}} \\
		X &&& {X_\gener{S}}
	\arrow["{(\bound{}{}e)_{e \in \gener{S}}}", from=1-1, to=2-1]
	\arrow["{(e)_{e \in \gener{S}}}", from=1-4, to=2-4]
	\arrow[hook, from=2-1, to=2-4]
	\arrow["{\coprod_{e \in \gener{S}}\molecin{\imath_e}}", hook, from=1-1, to=1-4]
	\arrow["\lrcorner"{anchor=center, pos=0.125, rotate=180}, draw=none, from=2-4, to=1-1]
\end{tikzcd}\]
in $\omegacat$, where, for each $e \in \gener{S}$, $U_e$ is an atom and $\imath_e \colon \bound{}{}U_e \incl U_e$ is the inclusion of its boundary.
\end{dfn}

\begin{comm}
	The functor $X \incl X_\gener{S}$ in a cellular extension is always injective, as shown in \cite[Section 4]{makkai2005word}.

	This is a non-standard definition of cellular extension, allowing any atom as a potential cell shape; the usual definition only uses \emph{globes}.
	However, the two are equivalent in the sense that a cellular extension in our sense can always be turned into a cellular extension in the more restrictive sense.
\end{comm}

\begin{dfn}[Polygraph]
	A \emph{polygraph}, also known as \emph{computad}, is a strict $\omega$-category $X$ together with, for each $n \in \mathbb{N}$, a pushout diagram
\[
	\begin{tikzcd}
		{\coprod_{e \in \grade{n}{\gener{S}}} \molecin{\bound{}{}U_e}} &&& {\coprod_{e \in \grade{n}{\gener{S}}} \molecin{U_e}} \\
		\skel{n-1}{X} &&& \skel{n}{X}
	\arrow["{(\bound{}{}e)_{e \in \grade{n}{\gener{S}}}}", from=1-1, to=2-1]
	\arrow["{(e)_{e \in \grade{n}{\gener{S}}}}", from=1-4, to=2-4]
	\arrow[hook, from=2-1, to=2-4]
	\arrow["{\coprod_{e \in \grade{n}{\gener{S}}}\molecin{\imath_e}}", hook, from=1-1, to=1-4]
	\arrow["\lrcorner"{anchor=center, pos=0.125, rotate=180}, draw=none, from=2-4, to=1-1]
\end{tikzcd}\]
in $\omegacat$, exhibiting $\skel{n}{X}$ as a cellular extension of $\skel{n-1}{X}$, such that $U_e$ is an $n$\nbd dimensional atom for all $e \in \grade{n}{\gener{S}}$.
The set
\[
	\gener{S} \eqdef \sum_{n \in \mathbb{N}} \set{e\isocl{\idd{U_e}} \mid e \in \grade{n}{\gener{S}}}
\]
is called the set of \emph{generating cells} of the polygraph.
We write $\cwcom{X}{S}$ for a polygraph $X$ with set $\gener{S}$ of generating cells.
\end{dfn}

\begin{lem} \label{lem:generating_cells_are_generating}
	Let $\cwcom{X}{S}$ be a polygraph.
	Then $\gener{S}$ is a basis for $X$.
\end{lem}
\begin{proof}
	The fact that $\gener{S}$ is a generating set and its minimality are consequences of \cite[Proposition 15.1.8 and Lemma 16.6.2]{ara2023polygraphs}, respectively.
\end{proof}

\begin{lem} \label{lem:frame_acyclic_cellular_extension}
	Let $P$ be an oriented graded poset, $n \in \mathbb{N}$, and let $\grade{n}{\gener{S}}$ be a set containing one pasting diagram
	\[
		e \equiv \molecin{e}\colon \molecin{U_e} \to \skel{n}{\molecin{P}}
	\]
	for each $\isocl{e\colon U_e \to P}$ in $\atomin{P}$ such that $\dim{U_e} = n$.
	If $\skel{n}{P}$ has frame-acyclic molecules, then
\[
	\begin{tikzcd}
		{\coprod_{e \in \grade{n}{\gener{S}}} \molecin{\bound{}{}U_e}} &&& {\coprod_{e \in \grade{n}{\gener{S}}} \molecin{U_e}} \\
		\skel{n-1}{\molecin{P}} &&& \skel{n}{\molecin{P}}
		\arrow["{(\bound{}{}e)_{e \in \grade{n}{\gener{S}}}}", from=1-1, to=2-1]
		\arrow["{(e)_{e \in \grade{n}{\gener{S}}}}", from=1-4, to=2-4]
		\arrow[hook, from=2-1, to=2-4]
		\arrow["{\coprod_{e \in \grade{n}{\gener{S}}}\molecin{\imath_e}}", hook, from=1-1, to=1-4]
		\arrow["\lrcorner"{anchor=center, pos=0.125, rotate=180}, draw=none, from=2-4, to=1-1]
\end{tikzcd}\]
	is a pushout diagram in $\omegacat$, exhibiting $\skel{n}{\molecin{P}}$ as a cellular extension of $\skel{n-1}{\molecin{P}}$.
\end{lem}
\begin{proof}
	Let $X$ be a strict $\omega$\nbd category and let
\[
	\begin{tikzcd}
		{\coprod_{e \in \grade{n}{\gener{S}}} \molecin{\bound{}{}U_e}} &&& {\coprod_{e \in \grade{n}{\gener{S}}} \molecin{U_e}} \\
		\skel{n-1}{\molecin{P}} &&& X
		\arrow["{(\bound{}{}e)_{e \in \grade{n}{\gener{S}}}}", from=1-1, to=2-1]
		\arrow["\ell", from=1-4, to=2-4]
		\arrow["h", from=2-1, to=2-4]
		\arrow["{\coprod_{e \in \grade{n}{\gener{S}}}\molecin{\imath_e}}", hook, from=1-1, to=1-4]
\end{tikzcd}\]
	be a commutative diagram of strict functors.
	We define $\overbar{h}\colon \skel{n}{\molecin{P}} \to X$ as follows on cells $\isocl{f\colon U \to P}$ in $\skel{n}{\molecin{P}}$.
	If $\dim{U} < n$, then we let
	\[	\overbar{h}\isocl{f} \eqdef h\isocl{f}. \]
	Suppose $\dim{U} = n$; we proceed by induction on $\lydim{U}$.
	If $\lydim{U} = -1$, then by Lemma \ref{lem:induction_on_layering_dimension} $U$ is an atom, so there exists a unique $\molecin{e} \in \grade{n}{\gener{S}}$ such that $\isocl{f} = \isocl{e}$, and we let $\overbar{h}\isocl{f} \eqdef \ell \isocl{\idd{U_e}}$.
	If $\lydim{U} = k \geq 0$, then $U$ admits a $k$\nbd layering $(\order{i}{U})_{i=1}^m$, and each layer $\order{i}{U}$ has strictly lower layering dimension.
	Then we let
	\[
		\overbar{h}\isocl{f} \eqdef \overbar{h}\isocl{\restr{f}{\order{1}{U}}} \cp{k} \ldots \cp{k} \overbar{h}\isocl{\restr{f}{\order{m}{U}}}.
	\]
	By construction, if $\overbar{h}$ is well-defined, then it is a strict functor satisfying $\overbar{h} \after (e)_{e \in \grade{n}{\gener{S}}} = \ell$ and restricting to $h$ on $\skel{n-1}{\molecin{P}}$.
	Moreover, let $h'$ be another strict functor with the same property.
	Then $h'$ agrees with $\overbar{h}$ on all atoms of dimension $\leq n$, which form a basis of $\skel{n}{\molecin{P}}$.
	It follows from Lemma \ref{lem:functors_equal_on_generating_set} that $h' = \overbar{h}$.
	It only remains to show that $\overbar{h}$ is well-defined, that is, it is independent of the choice of a $k$\nbd layering of $U$ when $\dim{U} = n$ and $k \eqdef \lydim{U} \geq 0$.

	We may assume, inductively, that $\overbar{h}$ is well-defined on all cells $\isocl{g\colon V \to P}$ such that $\dim{V} < n$ or $\lydim{V} < k$.
	Let $(\order{i}{U})_{i=1}^m$ and $(\order{i}{V})_{i=1}^m$ be two $k$\nbd layerings of $U$ and let $(\order{i}{x})_{i=1}^m$, $(\order{i}{y})_{i=1}^m$ be the induced $k$\nbd orderings.
	We now proceed as in the proof of Lemma
	\ref{lem:frdim_layerings_bijection_with_orderings}, letting $\sigma$ be the unique permutation such that $\order{i}{x} = \order{\sigma(i)}{y}$ for all $i \in \set{1, \ldots, m}$, letting
	\[
		d \eqdef \fun{d}((\order{i}{x})_{i=1}^m, (\order{i}{y})_{i=1}^m)
	\]
	be the number of pairs $(j, j')$ such that $j < j'$ but $\sigma(j') < \sigma(j)$, and proceeding by induction on $d$.
	If $d = 0$, then the two layerings are equal up to layer-wise isomorphism.
	If $d > 0$, then there exists $j < m$ such that $\sigma(j + 1) < \sigma(j)$, and we let $W \submol U$ be the image of $\order{j}{U} \cp{k} \order{j+1}{U}$ in $U$.
	Then $W$ contains exactly two elements $z_1 \eqdef \order{j}{x} = \order{\sigma(j)}{y}$ and $z_2 \eqdef \order{j+1}{x} = \order{\sigma(j+1)}{y}$ of dimension $> k$, yet there can be no edge between them in $\maxflow{k}{U}$, from which we deduce that $r \eqdef \frdim{W} < k$.
	By assumption, $W$ is frame-acyclic, so by Theorem
	\ref{thm:frame_acyclic_has_frame_layerings} there exists an $r$\nbd layering of $W$, hence also a pair of molecules $\order{1}{W}$, $\order{2}{W}$, each containing a single element of dimension $> k$, such that $W$ is isomorphic to $\order{1}{W} \cp{r} \order{2}{W}$.
	We may assume, without loss of generality, that $z_1$ is in the image of $\order{1}{W}$ and $z_2$ in the image of $\order{2}{W}$.
	We then have
	\begin{align*}
		& \overbar{h}\isocl{\restr{f}{\order{j}{U}}} \cp{k}
		\overbar{h}\isocl{\restr{f}{\order{j+1}{U}}} = \\
		& \quad = \left( \overbar{h}\isocl{\restr{f}{\order{1}{W}}} \cp{r}
		\overbar{h}\isocl{\restr{f}{\bound{k}{-}\order{2}{W}}} \right) \cp{k}
		\left( \overbar{h}\isocl{\restr{f}{\bound{k}{+}\order{1}{W}}} \cp{r}
		\overbar{h}\isocl{\restr{f}{\order{2}{W}}} \right),
	\end{align*}
	which by interchange and unitality in $X$ is equal to
	\begin{align*}
		& \overbar{h}\isocl{\restr{f}{\order{1}{W}}} \cp{r}
		\overbar{h}\isocl{\restr{f}{\order{2}{W}}} = \\
		& \quad = \left( \overbar{h}\isocl{\restr{f}{\bound{k}{-}\order{1}{W}}} \cp{r}
		\overbar{h}\isocl{\restr{f}{\order{2}{W}}} \right) \cp{k}
		\left( \overbar{h}\isocl{\restr{f}{\order{1}{W}}} \cp{r}
		\overbar{h}\isocl{\restr{f}{\bound{k}{+}\order{2}{W}}} \right) = \\
		& \quad = \overbar{h}\isocl{\restr{f}{\order{j}{\tilde{U}}}} \cp{k} \overbar{h}\isocl{\restr{f}{\order{j+1}{\tilde{U}}}},
	\end{align*}
	where we let $\order{j}{\tilde{U}} \eqdef \bound{k}{-}\order{1}{W} \cp{r} \order{2}{W}$ and $\order{j+1}{\tilde{U}} \eqdef \order{1}{W} \cp{r} \bound{k}{+} \order{2}{W}$.
	Notice that all the $n$\nbd dimensional cells in this calculation involve molecules whose layering dimension is $< k$, so $\overbar{h}$ is well-defined on each of them.
	Letting $\order{i}{\tilde{U}} \eqdef \order{i}{U}$ for all $i \not\in \set{j, j+1}$, we have that
	\begin{enumerate}
		\item $(\order{i}{\tilde{U}})_{i=1}^m$ is a $k$\nbd layering of $U$,
		\item the definition of $\overbar{h}\isocl{f}$ using $(\order{i}{U})_{i=1}^m$ is equal to the one using $(\order{i}{\tilde{U}})$, and
		\item the induced $k$\nbd ordering $(\order{i}{\tilde{x}})_{i=1}^m \eqdef (\order{1}{x}, \ldots, \order{j+1}{x}, \order{j}{x}, \ldots, \order{m}{x})$ satisfies $\fun{d}((\order{i}{\tilde{x}})_{i=1}^m, (\order{i}{y})_{i=1}^m) < d$,
	\end{enumerate}
	so, by the inductive hypothesis on $d$, the definition of $\overbar{h}\isocl{f}$ using $(\order{i}{\tilde{U}})_{i=1}^m$ is equal to the definition using $(\order{i}{V})_{i=1}^m$.
	We conclude that $\overbar{h}\isocl{f}$ is well-defined, which completes the proof.
\end{proof}

\begin{thm} \label{thm:frame_acyclic_presents_polygraphs}
	Let $P$ be an oriented graded poset with frame-acyclic molecules.
	Then $\molecin{P}$ is a polygraph whose set of generating cells is $\atomin{P}$.
\end{thm}
\begin{proof}
	If $P$ has frame-acyclic molecules, then $\skel{n}{P}$ has frame-acyclic molecules for all $n \in \mathbb{N}$.
	The statement then follows from Lemma \ref{lem:frame_acyclic_cellular_extension}.
\end{proof}

\begin{comm}
	In fact, by the roundness property of atoms, if $\molecin{P}$ is a polygraph, then it is a \emph{regular polygraph} in the sense of \cite{henry2018regular}.
\end{comm}

\begin{comm}
	Frame-acyclic molecules seem to be the tightest combinatorial condition ensuring that $\molecin{P}$ is a polygraph: the 4\nbd dimensional molecule of \cite[Example 126]{hadzihasanovic2023higher} does not present a polygraph, since its two possible 3\nbd layerings cannot be related by applications of the interchange equation.
	This does not mean that having frame-acyclic molecules is \emph{equivalent} to $\molecin{P}$ being a polygraph: one can engineer a variant of this example where a 4\nbd dimensional molecule only has one valid 3\nbd layering (even though it has multiple 3\nbd orderings), so there are no ``extra equations''.
	However the algebraic freeness of such examples can be seen as accidental.
\end{comm}

\section{Dimension-wise acyclicity and Steiner complexes} \label{sec:dwacy}

In the previouse section we saw that the property of frame-acyclic molecules is sufficient to obtain freeness of the $\omega$\nbd category presented by an oriented graded poset.
Unfortunately, beyond low dimensions where it holds automatically, this property is difficult to check, since one needs to verify acyclicity of one graph \emph{for each molecule} over an oriented graded poset.
Its stability properties are also unclear.
However, it is implied by stronger acyclicity conditions which are more restrictive, but easier both to check and to work with in practice.

In this section, we start by considering a property we call \emph{dimension-wise acyclicity}.
This is of interest because it corresponds in a precise sense to the loop-freeness property of \emph{augmented directed chain complexes} considered in \cite{steiner2004omega}; see also \cite{ara2020joint}.
This will allow us to make a precise connection between our framework and what has come to be known as \emph{Steiner theory}.

\begin{dfn}[Dimension-wise acyclic oriented graded poset]
Let $P$ be an oriented graded poset.
We say that $P$ is \emph{dimension-wise acyclic} if, for all $k \geq -1$, $\flow{k}{P}$ is acyclic.
\end{dfn}

\begin{prop} \label{prop:dimensionwise_implies_frame_acyclic}
Let $U$ be a molecule.
If $U$ is dimension-wise acyclic, then $U$ is frame-acyclic.
\end{prop}
\begin{proof}
Let $V \submol U$ be a submolecule inclusion, $r \eqdef \frdim{V}$.
If $U$ is dimension-wise acyclic, then $\flow{r}{U}$ is acyclic.
Then $\flow{r}{V}$ is the induced subgraph of $\flow{r}{U}$ on the vertices contained in $V$, while $\maxflow{r}{V}$ is an induced subgraph of $\flow{r}{V}$, and an induced subgraph of an acyclic graph is always acyclic.
\end{proof}

\begin{exm} \label{exm:non_dw_acy}
	This example, which is essentially \cite[Fig.~4]{steiner1993algebra}, shows that the implication of Proposition \ref{prop:dimensionwise_implies_frame_acyclic} is strict.
	Let $U$ be a 3\nbd dimensional atom whose input and output boundaries correspond to the pasting diagrams
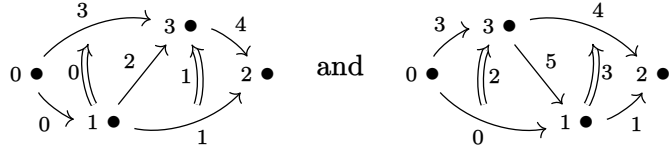
\[
\begin{tikzcd}[sep=tiny]
	&& {{\scriptstyle 3}\; \bullet} \\
	{{\scriptstyle 0}\; \bullet} &&& {{\scriptstyle 2}\; \bullet} \\
	& {{\scriptstyle 1}\; \bullet}
	\arrow[""{name=0, anchor=center, inner sep=0}, "1"', curve={height=12pt}, from=3-2, to=2-4]
	\arrow["0"', curve={height=6pt}, from=2-1, to=3-2]
	\arrow["4", curve={height=-6pt}, from=1-3, to=2-4]
	\arrow[""{name=1, anchor=center, inner sep=0}, "3", curve={height=-12pt}, from=2-1, to=1-3]
	\arrow["2", from=3-2, to=1-3]
	\arrow["1", curve={height=6pt}, shorten <=7pt, Rightarrow, from=0, to=1-3]
	\arrow["0", curve={height=-6pt}, shorten >=7pt, Rightarrow, from=3-2, to=1]
\end{tikzcd}
	\quad \text{and} \quad
	\begin{tikzcd}[sep=tiny]
	& {{\scriptstyle 3}\; \bullet} \\
	{{\scriptstyle 0}\; \bullet} &&& {{\scriptstyle 2}\; \bullet} \\
	&& {{\scriptstyle 1}\; \bullet}
	\arrow[""{name=0, anchor=center, inner sep=0}, "0"', curve={height=12pt}, from=2-1, to=3-3]
	\arrow["1"', curve={height=6pt}, from=3-3, to=2-4]
	\arrow["3", curve={height=-6pt}, from=2-1, to=1-2]
	\arrow[""{name=1, anchor=center, inner sep=0}, "4", curve={height=-12pt}, from=1-2, to=2-4]
	\arrow["5", from=1-2, to=3-3]
	\arrow["2"', curve={height=-6pt}, shorten <=7pt, Rightarrow, from=0, to=1-2]
	\arrow["3"', curve={height=6pt}, shorten >=7pt, Rightarrow, from=3-3, to=1]
\end{tikzcd}\]
	respectively.
	Let $(n, k)$ denote the $n$\nbd dimensional cell labelled $k$.
	Then $\flow{0}{U}$ contains the cycle
	\[
		(1, 2) \to (1, 5) \to (1, 2),
	\]
	so $U$ is not dimension-wise acyclic.
	However, since $U$ is 3\nbd dimensional, it is frame-acyclic as a consequence of Theorem
	\ref{thm:dim3_frame_acyclic}.
\end{exm}

\begin{prop} \label{prop:flow_graph_homomorphism}
	Let $f\colon P \to Q$ be a local embedding of oriented graded posets.
	Then, for all $k \geq -1$, $f$ induces a homomorphism $\flow{k}{f}\colon \flow{k}{P} \to \flow{k}{Q}$.
\end{prop}
\begin{proof}
	For all $x \in P$, the restriction $\restr{f}{\clset{x}}$ is an inclusion.
	By Proposition \ref{prop:inclusions_preserve_faces}, for all $\alpha \in \set{+, -}$, if $y \in \faces{k}{\alpha}x$, then $f(y) \in \faces{k}{\alpha}f(x)$, so if there is an edge between $x$ and $y$ in $\flow{k}{P}$, then there is an edge between $f(x)$ and $f(y)$ in $\flow{k}{Q}$.
\end{proof}

\begin{cor} \label{cor:local_embedding_into_dimensionwise_acyclic}
Let $f\colon P \to Q$ be a local embedding of oriented graded posets.
If $Q$ is dimension-wise acyclic, then so is $P$.
\end{cor}

\begin{prop} \label{prop:dw_acyclic_rdcpx_has_frame_acyclic_molecules}
	Let $P$ be a dimension-wise acyclic regular directed complex.
	Then $P$ has frame-acyclic molecules.
\end{prop}
\begin{proof}
	Let $U$ be a molecule and $f\colon U \to P$ a morphism.
	By Proposition \ref{prop:rdcpx_local_embeddings}, $f$ is a local embedding, so by Corollary \ref{cor:local_embedding_into_dimensionwise_acyclic} $U$ is dimension-wise acyclic.
	It follows from Proposition \ref{prop:dimensionwise_implies_frame_acyclic} that $U$ is frame-acyclic.
\end{proof}

\begin{cor} \label{cor:dw_acyclic_rdcpx_presents_polygraphs}
	Let $P$ be a dimension-wise acyclic regular directed complex.
	Then $\molecin{P}$ is a polygraph.
\end{cor}

\begin{dfn}[Augmented directed chain complex]
	An \emph{augmented chain complex} $C$ is a chain complex of abelian groups in non-negative degree
\[\begin{tikzcd}
	\ldots & {\chain{C}{n}} & {\chain{C}{n-1}} & \ldots & {\chain{C}{1}} & {\chain{C}{0}}
	\arrow["\der", from=1-1, to=1-2]
	\arrow["\der", from=1-2, to=1-3]
	\arrow["\der", from=1-3, to=1-4]
	\arrow["\der", from=1-4, to=1-5]
	\arrow["\der", from=1-5, to=1-6]
\end{tikzcd}\]
	together with a homomorphism $\eau\colon \chain{C}{0} \to \mathbb{Z}$ satisfying $\eau \after \der = 0$.
	A \emph{direction} on $C$ is a choice of a commutative submonoid $\grade{n}{\dir{C}}$ of $\grade{n}{C}$ for each $n \in \mathbb{N}$.
	An \emph{augmented directed chain complex} is an augmented chain complex $C$ together with a direction on its underlying chain complex.
\end{dfn}

\begin{dfn}[Homomorphism of augmented directed chain complexes] \index{directed chain complex!homomorphism}
	Let $C, D$ be augmented directed chain complexes.
	A \emph{homomorphism} $f\colon C \to D$ is a homomorphism of the underlying augmented chain complexes, that is, a sequence $(\chain{f}{n}\colon \chain{C}{n} \to \chain{D}{n})_{n \in \mathbb{N}}$ of homomorphisms of abelian groups satisfying
	\[
		\der \after \grade{n+1}{f} = \grade{n}{f} \after \der, \quad \quad e \after \grade{0}{f} = e,
	\]
	which is compatible with directions in the sense that
	\[
		\grade{n}{f}(\grade{n}{\dir{C}}) \subseteq \grade{n}{\dir{D}}
	\]
	for all $n \in \mathbb{N}$.
	Augmented directed chain complexes with their homomorphisms form a category $\dchaug$.
\end{dfn}

\begin{dfn}[Linearisation of a strict $\omega$-category] \index{strict $\omega$-category!linearisation}
	Let $X$ be a strict $\omega$-category.
	The \emph{linearisation of $X$} is the augmented directed chain complex $\linea{X}$ whose underlying augmented chain complex is defined by
	\[
		\grade{n}{\linea{X}} \eqdef \frac{
		\freeab{(\skel{n}{X})} } {
		\spanset{ t \cp{k} u - t - u \mid \text{$t, u \in \skel{n}{X}$, $k < n$} } }\;
	\]
	for all $n \in \mathbb{N}$, where $\freeab{(\skel{n}{X})}$ denotes the free abelian group on the set of cells of the $n$\nbd skeleton of $X$, with the homomorphisms determined by
	\[
		\der\colon t \in \skel{n}{X} \mapsto (\bound{n-1}{+}t - \bound{n-1}{-}t), \quad \quad 
		\eau\colon t \in \skel{0}{X} \mapsto 1
	\]
	for each $n > 0$, and the direction defined by
	\[
		\grade{n}{\dir{\linea{X}}} \eqdef \Ima \left(
			\freemon{(\skel{n}{X})}
			\incl \freeab{(\skel{n}{X})}
			\to \grade{n}{\linea{X}} \right)
	\]
	for each $n \in \mathbb{N}$, where $\freeab{(\skel{n}{X})} \to \grade{n}{\linea{X}}$ is the canonical quotient homomorphism.
	This extends to a functor $\omegacat \to \dchaug$, see \cite[Definition 2.4]{steiner2004omega}.
\end{dfn}

\begin{dfn}[Globular table in an augmented directed chain complex] \index{directed chain complex!globular table}
	Let $C$ be an augmented directed chain complex.
	A \emph{globular table in $C$} is a double sequence
	\[
		x \equiv (\gltab{n}{\alpha}{x})_{n \in \mathbb{N}, \, \alpha \in \set{+, -}}
	\]
	such that
	\begin{enumerate}
		\item $\gltab{n}{\alpha}{x} \in \grade{n}{\dir{C}}$ for all $n \in \mathbb{N}$ and $\alpha \in \set{+, -}$,
		\item $\der \gltab{n}{\alpha}{x} = \gltab{n-1}{+}{x} - \gltab{n-1}{-}{x}$ for all $n > 0$ and $\alpha \in \set{+, -}$,
		\item $\eau \gltab{0}{\alpha}{x} = 1$ for all $\alpha \in \set{+, -}$,
		\item there exists $m \in \mathbb{N}$ such that $\gltab{n}{\alpha}{x} = 0$ for all $n > m$ and $\alpha \in \set{+, -}$.
	\end{enumerate}
\end{dfn}

\begin{dfn}[Strict $\omega$-category of globular tables] \index{directed chain complex!globular table!strict $\omega$-category}
	Let $C$ be an augmented directed chain complex.
	The \emph{strict $\omega$-category of globular tables in $C$} is the strict $\omega$\nbd category $\nufun{C}$ whose set of cells is $\set{ x \mid \text{$x$ is a globular table in $C$} }$,
	with the boundary operators defined, for all $n \in \mathbb{N}$ and $\alpha \in \set{+, -}$, by
	\[
		\gltab{m}{\beta}{(\bound{n}{\alpha}x)} \eqdef
		\begin{cases}
			\gltab{m}{\beta}x
			& \text{if $m < n$,} \\
			\gltab{n}{\alpha}x
			& \text{if $m = n$,} \\
			0
			& \text{if $m > n$,}
		\end{cases}
	\]
	and the $k$\nbd composition operations defined, for all $k \in \mathbb{N}$ and $k$\nbd composable pairs $x, y$ of globular tables, by
	\[
		\gltab{n}{\alpha}{(x \cp{k} y)} \eqdef \gltab{n}{\alpha}{x} - \gltab{n}{\alpha}{(\bound{k}{+}x)} + \gltab{n}{\alpha}{y}.
	\]
	This extends to a functor $\nufun{}\colon \dchaug \to \omegacat$, see \cite[Definition 2.8]{steiner2004omega}.
\end{dfn}

\begin{prop} \label{prop:adjunction_dchaug_omegacat}
	The functor $\linea{}$ is left adjoint to $\nufun{}$.
\end{prop}
\begin{proof}
	See \cite[Theorem 2.11]{steiner2004omega}.
\end{proof}

\begin{dfn}[Augmented directed chain complex of an oriented thin graded poset]
	Let $P$ be an oriented graded poset such that $\augm{P}$ is oriented thin.
	The \emph{augmented directed chain complex of $P$}, denoted by $\dfreeab{P}$, is the augmented chain complex
\[\begin{tikzcd}
	\ldots & {\freeab{\grade{n}{P}}} & {\freeab{\grade{n-1}{P}}} & \ldots & {\freeab{\grade{1}{P}}} & {\freeab{\grade{0}{P}}}
	\arrow["\der", from=1-1, to=1-2]
	\arrow["\der", from=1-2, to=1-3]
	\arrow["\der", from=1-3, to=1-4]
	\arrow["\der", from=1-4, to=1-5]
	\arrow["\der", from=1-5, to=1-6]
\end{tikzcd}\]
where $\freeab{\grade{n}{P}}$ is the free abelian group on the set $\grade{n}{P}$ and, for each $n > 0$, the homomorphism $\der\colon \freeab{\grade{n}{P}} \to \freeab{\grade{n-1}{P}}$ is defined on the generators $x \in \grade{n}{P}$ by
\begin{equation} \label{eq:boundary_in_chain_complex}
	x \mapsto \sum_{y \in \faces{}{+}x} y - \sum_{y \in \faces{}{-}x} y,
\end{equation}
together with the homomorphism $\eau\colon \freeab{\grade{0}{P}} \to \mathbb{Z}$ defined on the generators $x \in \grade{0}{P}$ by $x \mapsto 1$ and the direction given by $\grade{n}{\dir{\freeab{P}}} \eqdef \freemon{\grade{n}{P}}$ for each $n \in \mathbb{N}$.
\end{dfn}

\begin{prop} \label{lem:augmented_chain_complex_is_indeed_one}
	Let $P$ be an oriented graded poset such that $\augm{P}$ is oriented thin.
	Then $\dfreeab{P}$ is well-defined as an augmented directed chain complex.
	Moreover, if $f\colon P \to Q$ is a morphism of oriented graded posets such that $\augm{P}$, $\augm{Q}$ are oriented thin, then the sequence of homomorphisms
	\begin{align*}
		\freeab{\grade{n}{f}}\colon \freeab{\grade{n}{P}} & \to \freeab{\grade{n}{Q}}, \\
		x \in \grade{n}{P} & \mapsto f(x) \in \grade{n}{Q}
	\end{align*}
	is a homomorphism $\dfreeab{f}\colon \dfreeab{P} \to \dfreeab{Q}$ of augmented directed chain complexes.
\end{prop}
\noindent
By Proposition \ref{prop:oriented_diamond_rdc}, this gives us an assignment $\dfreeab{}\colon \rdcpx \to \dchaug$ which is easily determined to be functorial.
There are now two ways to get from regular directed complexes to augmented directed chain complexes: either directly via $\dfreeab{}$, or by first passing to the $\omega$\nbd category of molecules, then applying Steiner's linearisation functor.
Fortunately, the two coincide up to natural isomorphism.

\begin{prop} \label{thm:two_functors_from_rdcpx_to_dchaug}
	Let $P$ be a regular directed complex.
	Then the assignment, for each $n \in \mathbb{N}$,
	\begin{align*}
		\grade{n}{\varphi}\colon \grade{n}{\freeab{P}} & \to
		\grade{n}{\left(\linea{\molecin{P}}\right)}, \\
		x \in \grade{n}{P} & \mapsto \isocl{\clset{x} \incl P} \in \skel{n}{\molecin{P}}
	\end{align*}
	is a natural isomorphism between $\dfreeab{P}$ and $\linea{\molecin{P}}$. 
\end{prop}

\noindent
Dually, there are two ways to get from a regular directed complex to an $\omega$\nbd category: either via the $\omega$\nbd category of molecules, or applying the right adjoint $\nufun{}$ after $\dfreeab{}$.
These two do \emph{not}, in general, coincide.
The rest of this section is dedicated to showing that they do coincide when $P$ is dimension-wise acyclic, in which case $\dfreeab{P}$ is a \emph{Steiner complex} in the sense of \cite{ara2020joint}.

\begin{dfn}[Basis of an augmented directed chain complex] \index{directed chain complex!basis} \index{basis!of an augmented directed chain complex}
	Let $C$ be an augmented directed chain complex.
	A \emph{basis for $C$} is a sequence of subsets $(\grade{n}{\gener{B}} \subseteq \grade{n}{C})_{n \in \mathbb{N}}$ such that, for all $n \in \mathbb{N}$, $\grade{n}{C}$ is isomorphic to $\freeab{\grade{n}{\gener{B}}}$ and $\grade{n}{\dir{C}}$ to $\freemon{\grade{n}{\gener{B}}}$.
\end{dfn}

\begin{dfn}[Support of a chain]
	Let $C$ be an augmented directed chain complex with basis $(\grade{n}{\gener{B}})_{n \in \mathbb{N}}$, $n \in \mathbb{N}$, and $x \equiv \sum_{b \in \grade{n}{\gener{B}}} x_b b \in \grade{n}{C}$.
	The \emph{support of $x$} is the subset
	\[
		\supp{x} \eqdef \set{b \in \grade{n}{\gener{B}} \mid x_b \neq 0} \subseteq \grade{n}{\gener{B}}.
	\]
\end{dfn}

\noindent
If $C$ has a basis, then for all $x$ there exist unique $x^+, x^- \in \grade{n}{\dir{C}}$ such that $x = x^+ - x^-$ and $\supp{x^+} \cap \supp{x^-} = \varnothing$.

\begin{dfn}[Unital basis]
	Let $C$ be an augmented directed chain complex with basis $(\grade{n}{\gener{B}})_{n \in \mathbb{N}}$.
	For all $n \in \mathbb{N}$ and $b \in \grade{n}{\gener{B}}$, let
	\[
		\gltab{m}{\alpha}{\batom{b}} \eqdef
		\begin{cases}
			0
			& \text{if $m > n$,} \\
			b
			& \text{if $m = n$,} \\
			(\der \gltab{m+1}{\alpha}{\batom{b}})^\alpha
			& \text{if $m < n$}
		\end{cases}
	\]
	for each $m \in \mathbb{N}$ and $\alpha \in \set{+, -}$, where the definition is obtained by downward recursion when $m \leq n$.
	We say that the basis $(\grade{n}{\gener{B}})_{n \in \mathbb{N}}$ is \emph{unital} if, for all $n \in \mathbb{N}$ and $b \in \grade{n}{\gener{B}}$,
	\[
		\batom{b} \equiv (\gltab{m}{\alpha}{\batom{b}})_{m \in \mathbb{N}, \, \alpha \in \set{+, -}}
	\]
	is a globular table, or, equivalently, if $\eau \gltab{0}{+}{\batom{b}} = \eau \gltab{0}{-}{\batom{b}} = 1$.
\end{dfn}

\begin{dfn}[Flow graph of an augmented directed chain complex with basis]
	Let $C$ be an augmented directed chain complex with basis $(\grade{n}{\gener{B}})_{n \in \mathbb{N}}$, $k \in \mathbb{N}$.
	The \emph{$k$\nbd flow graph of $C$} is the directed graph $\flow{k}{C}$ whose
	\begin{itemize}
		\item set of vertices is $\bigcup_{i > k} \grade{i}{\gener{B}}$, and
		\item set of edges is $\set{(b, c) \mid \supp{\gltab{k}{+}{\batom{b}}} \cap \supp{\gltab{k}{-}{\batom{c}}} \neq \varnothing}$,
	where the source of $(b, c)$ is $b$ and the target is $c$.
	\end{itemize}
\end{dfn}

\begin{dfn}[Steiner complex]
	A \emph{Steiner complex} is an augmented directed chain complex $C$ with a unital basis such that, for all $k \in \mathbb{N}$, $\flow{k}{C}$ is acyclic.
	We let $\stcpx$ denote the full subcategory of $\dchaug$ on the Steiner complexes.
\end{dfn}

\noindent
The following is the fundamental theorem of Steiner theory.

\begin{thm} \label{thm:steiner_main_theorem}
	The restriction of $\nufun{}\colon \dchaug \to \omegacat$ to $\stcpx$ is full and faithful.
	Moreover, if $C$ is a Steiner complex with basis $(\grade{n}{\gener{B}})_{n \in \mathbb{N}}$, then $\nufun{C}$ is a polygraph whose set of generating cells is
	\[
		\set{\batom{b} \mid b \in \bigcup_{n \in \mathbb{N}} \grade{n}{\gener{B}}}.
	\]
\end{thm}
\begin{proof}
	See \cite[Theorem 5.6 and Theorem 6.1]{steiner2004omega}.
\end{proof}

\noindent
We will need the following result, which we state without proof.

\begin{lem} \label{lem:chain_complex_boundary_of_molecule}
	Let $P$ be a regular directed complex and let $U \subseteq P$ be a molecule, $n \eqdef \dim{U} > 0$.
	Then, in $\dfreeab{P}$,
	\[
		\der \left( \sum_{x \in \grade{n}{U}} x \right) = \sum_{y \in \faces{}{+}U} y - \sum_{y \in \faces{}{-}U} y.
	\]
\end{lem}

\begin{lem} \label{lem:basis_atoms_in_dchaug}
	Let $P$ be a regular directed complex, $x \in P$, $m \in \mathbb{N}$, and $\alpha \in \set{+, -}$.
	Then, in $\dfreeab{P}$,
	\[
		\gltab{m}{\alpha}{\batom{x}} = \sum_{y \in \faces{m}{\alpha}x} y.
	\]
\end{lem}
\begin{proof}
	Let $n \eqdef \dim{x}$, so $x \in \grade{n}{P}$.
	By definition, for $m > n$, $\gltab{m}{\alpha}{\batom{x}} = 0$, while $\faces{m}{\alpha}x = \varnothing$, and the equality holds.
	For $m \leq n$, we proceed by downward recursion.
	If $m = n$, we have $\gltab{m}{\alpha}{\batom{x}} = x$, while $\faces{m}{\alpha}x = \set{x}$, and the equality holds.
	Let $m < n$.
	Then
	\[
		\der \gltab{m+1}{\alpha}{\batom{x}} = \der \left( \quad
		\sum_{\mathclap{y \in \faces{m+1}{\alpha} x}} \; y \; \right)
	\]
	by the inductive hypothesis, and $\faces{m+1}{\alpha} x = \grade{m+1}{(\bound{m+1}{\alpha}x)}$.
	By Lemma \ref{lem:chain_complex_boundary_of_molecule} and globularity of $\clset{x}$, this is equal to
	\[
		\sum_{\mathclap{y \in \faces{}{+}(\bound{m+1}{\alpha}x)}}\; y \quad - \;\quad
		\sum_{\mathclap{y \in \faces{}{-}(\bound{m+1}{\alpha}x)}}\; y \quad = \quad
		\sum_{\mathclap{y \in \faces{m}{+}x}}\; y \; - \quad
		\sum_{\mathclap{y \in \faces{m}{-}x}}\; y \;,
	\]
	hence by definition
	\[
		\gltab{m}{+}{\batom{x}} = \; \sum_{\mathclap{y \in \faces{m}{+}x}}\; y\,, \quad \quad
		\gltab{m}{-}{\batom{x}} = \; \sum_{\mathclap{y \in \faces{m}{-}x}}\; y\,.
	\]
	This completes the proof.
\end{proof}

\begin{cor} \label{cor:support_is_faces}
	Let $P$ be a regular directed complex, $x \in P$, $n \in \mathbb{N}$, $\alpha \in \set{+, -}$.
	Then $\supp{\gltab{n}{\alpha}{\batom{x}}} = \faces{n}{\alpha}x$.
\end{cor}

\begin{prop} \label{prop:rdcpx_unital_basis}
	Let $P$ be a regular directed complex.
	Then $(\grade{n}{P})_{n \in \mathbb{N}}$ is a unital basis of $\dfreeab{P}$.
\end{prop}
\begin{proof}
	Let $x \in P$.
	For all $\alpha \in \set{+, -}$, $\faces{0}{\alpha}x = \bound{0}{\alpha}x = \set{x^\alpha}$ for a unique $x^\alpha \in \grade{0}{P}$, since the point is the only 0\nbd dimensional molecule.
	Then, by Lemma \ref{lem:basis_atoms_in_dchaug}, $\eau \gltab{0}{\alpha}{\batom{x}} = \eau x^\alpha = 1$, so the basis $(\grade{n}{P})_{n \in \mathbb{N}}$ is unital.
\end{proof}

\begin{lem} \label{lem:flow_graph_of_dchaug_is_flow_graph_of_rdcpx}
	Let $P$ be a regular directed complex.
	Then $\flow{k}{\dfreeab{P}}$ is isomorphic to $\flow{k}{P}$.
\end{lem}
\begin{proof}
	Immediate from Corollary \ref{cor:support_is_faces}.
\end{proof}

\begin{prop} \label{prop:dw_acyclic_rdcpx_gives_steiner_complex}
	Let $P$ be a dimension-wise acyclic regular directed complex.
	Then $\dfreeab{P}$ is a Steiner complex.
\end{prop}
\begin{proof}
	Follows from Proposition \ref{prop:rdcpx_unital_basis} and Lemma \ref{lem:flow_graph_of_dchaug_is_flow_graph_of_rdcpx}.
\end{proof}

\begin{thm} \label{thm:two_omegacats_from_dw_acyclic_rdcpx}
	Let $P$ be a dimension-wise acyclic regular directed complex.
	Then $\nufun{\dfreeab{P}}$ is naturally isomorphic to $\molecin{P}$.
\end{thm}
\begin{proof}
	Composing the component $\eta\colon \molecin{P} \to \nufun{\linea{\molecin{P}}}$ of the unit of the adjunction between $\linea{}$ and $\nufun{}$ with the natural isomorphism between $\linea{\molecin{P}}$ and $\dfreeab{P}$ from Theorem \ref{thm:two_functors_from_rdcpx_to_dchaug}, we obtain a strict functor
	\[
		\varphi\colon \molecin{P} \to \nufun{\dfreeab{P}}.
	\]
	By Corollary \ref{cor:dw_acyclic_rdcpx_presents_polygraphs}, $\molecin{P}$ is a polygraph generated by $\set{\isocl{\clset{x} \incl P} \mid x \in P}$, while by Theorem \ref{thm:steiner_main_theorem} combined with Proposition \ref{prop:dw_acyclic_rdcpx_gives_steiner_complex}, $\nufun{\dfreeab{P}}$ is a polygraph whose set of generating cells is $\set{\batom{x} \mid x \in P}$.
By sending $\isocl{\clset{x} \incl P}$ to $\batom{x}$, $\varphi$ determines a bijection between the generating cells of $\molecin{P}$ and of $\nufun{\dfreeab{P}}$.
	By \cite[Proposition 16.2.12]{ara2023polygraphs}, we conclude that $\varphi$ is an isomorphism of polygraphs.
\end{proof}

\begin{exm}
	Theorem \ref{thm:two_omegacats_from_dw_acyclic_rdcpx} does not extend beyond dimension-wise acyclic regular directed complexes.
	Let $P$ be the regular directed complex encoding the 1\nbd dimensional diagram
\begin{equation}\begin{tikzcd} \label{eq:nonacyclic_graph}
	{{\scriptstyle a} \;\bullet} && {{\scriptstyle b} \;\bullet}
	\arrow["f", from=1-1, to=1-3]
	\arrow["g"', curve={height=-18pt}, from=1-3, to=1-1]
	\arrow["h"', curve={height=18pt}, from=1-3, to=1-1]
\end{tikzcd}\end{equation}
	which is evidently not dimension-wise acyclic.
	Then $\molecin{P}$ is isomorphic to the free category on the directed graph (\ref{eq:nonacyclic_graph}).
	However, in $\nufun{\dfreeab{P}}$, let
	\[
		x \eqdef \batom{f} \cp{0} \batom{g}, \quad \quad
		y \eqdef \batom{f} \cp{0} \batom{h},
	\]
	which as globular tables are defined, for all $\alpha \in \set{+, -}$, by
	\[
		\gltab{n}{\alpha}x \eqdef \begin{cases}
			a & \text{if $n = 0$,} \\
			f + g & \text{if $n = 1$}, \\
			0 & \text{if $n > 1$},
		\end{cases}
		\quad
		\gltab{n}{\alpha}y \eqdef \begin{cases}
			a & \text{if $n = 0$,} \\
			f + h & \text{if $n = 1$}, \\
			0 & \text{if $n > 1$}.
		\end{cases}
	\]
	Then $x \cp{0} y$ and $y \cp{0} x$ are both equal to the globular table $z$ defined, for all $\alpha \in \set{+, -}$, by
	\[
		\gltab{n}{\alpha}z \eqdef
		\begin{cases}
			a & \text{if $n = 0$}, \\
			2f + g + h & \text{if $n = 1$}, \\
			0 & \text{if $n > 1$}.
		\end{cases}
	\]
	We conclude that $\nufun{\dfreeab{P}}$ is not free, so it is not isomorphic to $\molecin{P}$.
\end{exm}

\section{Stronger acyclicity conditions} \label{sec:stronger}

While dimension-wise acyclicity is a more manageable sufficient condition for frame-acyclicity, it does not guarantee a second property that we considered in the introduction, that is, that the $\omega$\nbd category of molecules over $P$ consists only of \emph{subsets} of $P$.

\begin{exm}
Let $U$ be the 2\nbd dimensional molecule encoding the shape of the pasting diagram
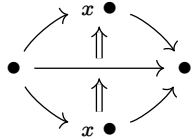
\[\begin{tikzcd}[sep=small]
	& {{\scriptstyle x}\;\bullet} \\
	\bullet && \bullet \\
	& {{\scriptstyle x}\; \bullet}
	\arrow[curve={height=6pt}, from=2-1, to=3-2]
	\arrow[""{name=0, anchor=center, inner sep=0}, from=2-1, to=2-3]
	\arrow[curve={height=6pt}, from=3-2, to=2-3]
	\arrow[curve={height=-6pt}, from=2-1, to=1-2]
	\arrow[curve={height=-6pt}, from=1-2, to=2-3]
	\arrow[shorten >=3pt, Rightarrow, from=3-2, to=0]
	\arrow[shorten <=3pt, Rightarrow, from=0, to=1-2]
\end{tikzcd}\]
and let $P$ be the result of identifying the two 0\nbd dimensional cells marked with $x$.
Then $P$ is a dimension-wise acyclic regular directed complex, and the canonical quotient map $q\colon U \to P$ is a molecule over $P$.
However, $q$ is evidently not injective.
\end{exm}

In this section, following \cite{steiner1993algebra}, we consider a strengthening of dimension-wise acyclicity which does guarantee this property at least for regular directed complexes.
Then, we consider an even stronger acyclicity property, relying on acyclicity of a \emph{single} directed graph, which, as we will see in Section \ref{sec:stability}, has better stability properties with respect to a number of constructions.

\begin{dfn}[Extended flow graph]
Let $P$ be an oriented graded poset, $k \geq -1$.
The \emph{extended $k$\nbd flow graph of $P$} is the bipartite directed graph $\extflow{k}{P}$ whose
\begin{itemize}
	\item set of vertices is
		\[
			P = \bigcup_{i \leq k} \grade{i}{P} + \bigcup_{i > k} \grade{i}{P},
		\]
	\item set of edges is $E_- + E_+$, where
		\begin{align*}
			E_- & \eqdef \set{ (y, x) \mid y \in \bigcup_{i \leq k}\grade{i}{P}, x \in \bigcup_{i > k} \grade{i}{P}, y \in \inter{\bound{k}{-}x} }, \\
			E_+ & \eqdef \set{ (y, x) \mid y \in \bigcup_{i > k}\grade{i}{P}, x \in \bigcup_{i \leq k}\grade{i}{P}, x \in \inter{\bound{k}{+}y} },
		\end{align*}
	where the source of $(y, x)$ is $y$ and the target is $x$.
\end{itemize}
\end{dfn}

\begin{dfn}[Strongly dimension-wise acyclic oriented graded poset]
Let $P$ be an oriented graded poset.
We say that $P$ is \emph{strongly dimension-wise acyclic} if, for all $k \geq -1$, $\extflow{k}{P}$ is acyclic.
\end{dfn}

\begin{comm}
	Strong dimension-wise acyclicity is essentially the same as \emph{loop-freeness} in the sense of \cite{steiner1993algebra}.
\end{comm}

\begin{lem} \label{lem:path_in_flow_induces_path_in_extflow}
	Let $P$ be an oriented graded poset, $k \geq -1$, and suppose $x, y \in \bigcup_{i > k}\grade{i}{P}$.
	If there exists a path from $x$ to $y$ in $\flow{k}{P}$, then there exists a path from $x$ to $y$ in $\extflow{k}{P}$.
\end{lem}
\begin{proof}
	Consider a path $x = x_0 \to x_1 \to \ldots \to x_m \to y$ from $x$ to $y$ in $\flow{k}{P}$.
	By definition of the $k$\nbd flow graph, for all $i \in \set{1, \ldots, m}$, there exists $z_i \in \faces{k}{+}x_{i-1} \cap \faces{k}{-}x_i$.
	By definition of the extended $k$\nbd flow graph, there exist edges $x_{i-1} \to z_i$ and $z_i \to x_i$ in $\extflow{k}{P}$.
	Concatenating all the two-step paths $x_{i-1} \to z_i \to x_i$, we obtain a path from $x$ to $y$ in $\extflow{k}{P}$.
\end{proof}

\begin{prop} \label{prop:strong_dw_implies_dw}
	Let $P$ be a strongly dimension-wise acyclic oriented graded poset.
	Then $P$ is dimension-wise acyclic.
\end{prop}
\begin{proof}
	By Lemma \ref{lem:path_in_flow_induces_path_in_extflow} a cycle in $\flow{k}{U}$ induces a cycle in $\extflow{k}{U}$.
\end{proof}

\begin{prop} \label{prop:extflow_graph_homomorphism}
	Let $f\colon P \to Q$ be a local embedding of oriented graded posets.
	For all $k \geq -1$, $f$ induces a homomorphism $\extflow{k}{f}\colon \extflow{k}{P} \to \extflow{k}{Q}$.
\end{prop}
\begin{proof}
	Similar to the proof of Proposition \ref{prop:flow_graph_homomorphism}, using the fact that, by Proposition \ref{prop:inclusions_preserve_faces}, inclusions preserve boundaries, hence they also preserve interiors.
\end{proof}

\begin{cor} \label{cor:local_embedding_into_strong_dw_acyclic}
Let $f\colon P \to Q$ be a local embedding of oriented graded posets.
If $Q$ is strongly dimension-wise acyclic, then so is $P$.
\end{cor}

\begin{lem} \label{lem:extflow_connected_frame_acyclic}
Let $U$ be a frame-acyclic molecule, $x, y \in U$.
Then there exists $k \geq -1$ such that there is a path from $x$ to $y$ or a path from $y$ to $x$ in $\extflow{k}{U}$.
\end{lem}
\begin{proof}
	See \cite[Theorem 2.16]{steiner1993algebra}, which applies by Remark \ref{rmk:split_is_frameacy}.
\end{proof}

\begin{prop} \label{prop:molecule_over_strongly_dimensionwise_acyclic}
	Let $U$ be a molecule, $P$ a strongly dimension-wise acyclic oriented graded poset, and $f\colon U \to P$ a local embedding.
	Then $f$ is an inclusion.
\end{prop}
\begin{proof}
Let $x, y \in U$ and suppose that $f(x) = f(y)$.
By Corollary \ref{cor:local_embedding_into_dimensionwise_acyclic}, $U$ is strongly dimension-wise acyclic.
It follows from Proposition \ref{prop:strong_dw_implies_dw} and \ref{prop:dimensionwise_implies_frame_acyclic} that $U$ is frame acyclic, so by Lemma \ref{lem:extflow_connected_frame_acyclic} there exists $k \geq -1$ such that there is a path from $x$ to $y$ or a path from $y$ to $x$ in $\extflow{k}{U}$.
Then by Proposition \ref{prop:extflow_graph_homomorphism} $\extflow{k}{f}$ maps this onto a cycle in $\extflow{k}{P}$, a contradiction, unless $x = y$ and the path is constant.
We conclude that $f$ is injective.
\end{proof}

\begin{cor} \label{cor:strongly_dimensionwise_acyclic_omega_cat}
Let $P$ be a strongly dimension-wise acyclic regular directed complex.
Then
\[
	\molecin{P} = \set{ \isocl{ U \incl P } \mid U \subseteq P, \text{$U$ is a molecule} }. 
\]
\end{cor}
\begin{proof}
	Follows from Proposition \ref{prop:molecule_over_strongly_dimensionwise_acyclic} together with Proposition \ref{prop:rdcpx_local_embeddings}.
\end{proof}

\begin{rmk}
	In particular, if $P$ is finite and strongly dimension-wise acyclic, it follows that $\molecin{P}$ has finitely many cells.
\end{rmk}

\begin{dfn}[Acyclic oriented graded poset]
Let $P$ be an oriented graded poset.
We say that $P$ is \emph{acyclic} if $\hasseo{P}$ is acyclic.
\end{dfn}

\begin{comm}
	Acyclicity is essentially the same as \emph{total loop-freeness} in \cite{steiner1993algebra}.
	As we will see, it is also related to \emph{strong loop-freeness} in \cite{steiner2004omega}.
\end{comm}

\begin{prop} \label{prop:acyclicity_implies_strong_dw_acyclicity}
	Let $P$ be an acyclic regular directed complex, $x, y \in P$, and $k \geq -1$.
	If there is a path from $x$ to $y$ in $\extflow{k}{P}$, then there is a path from $x$ to $y$ in $\hasseo{P}$.
	Consequently, $P$ is strongly dimension-wise acyclic.
\end{prop}
\begin{proof}
	See \cite[Proposition 2.15 and Proposition 5.2]{steiner1993algebra}.
\end{proof}

\begin{rmk}
	In fact, the following, stronger fact holds: in \emph{every} regular directed complex $P$, if there is a path from $x$ to $y$ in $\extflow{k}{P}$, then there is a path from $x$ to $y$ in $\hasseo{P}$.
	However, the proof is quite lengthy and technical and does not add much for our purposes.
\end{rmk}

\begin{exm} \label{exm:non_acyclic}
	Let $U$ be a 3\nbd dimensional atom whose input and output boundaries encode the pasting diagrams
	
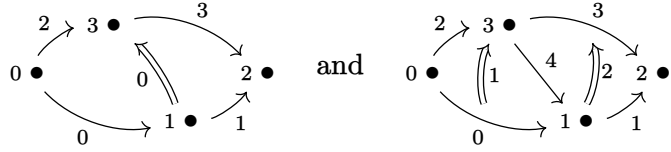
\[\begin{tikzcd}[sep=tiny]
	& {{\scriptstyle 3}\; \bullet} \\
	{{\scriptstyle 0}\; \bullet} &&& {{\scriptstyle 2}\; \bullet} \\
	&& {{\scriptstyle 1}\; \bullet}
	\arrow["0"', curve={height=12pt}, from=2-1, to=3-3]
	\arrow["1"', curve={height=6pt}, from=3-3, to=2-4]
	\arrow["2", curve={height=-6pt}, from=2-1, to=1-2]
	\arrow["3", curve={height=-12pt}, from=1-2, to=2-4]
	\arrow["0", curve={height=6pt}, Rightarrow, from=3-3, to=1-2]
\end{tikzcd} \quad \text{and} \quad
\begin{tikzcd}[sep=tiny]
	& {{\scriptstyle 3}\; \bullet} \\
	{{\scriptstyle 0}\; \bullet} &&& {{\scriptstyle 2}\; \bullet} \\
	&& {{\scriptstyle 1}\; \bullet}
	\arrow[""{name=0, anchor=center, inner sep=0}, "0"', curve={height=12pt}, from=2-1, to=3-3]
	\arrow["1"', curve={height=6pt}, from=3-3, to=2-4]
	\arrow["2", curve={height=-6pt}, from=2-1, to=1-2]
	\arrow[""{name=1, anchor=center, inner sep=0}, "3", curve={height=-12pt}, from=1-2, to=2-4]
	\arrow["4", from=1-2, to=3-3]
	\arrow["1"', curve={height=-6pt}, shorten <=7pt, Rightarrow, from=0, to=1-2]
	\arrow["2"', curve={height=6pt}, shorten >=7pt, Rightarrow, from=3-3, to=1]
\end{tikzcd}\]
	respectively, and let $(n, k)$ denote the $n$\nbd dimensional cell labelled with $k$.
	Then the extended 0\nbd flow graph $\extflow{0}{U}$ is
	
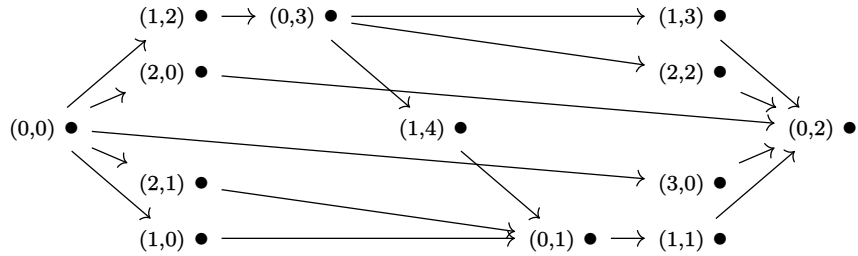
\[\begin{tikzcd}[column sep=small, row sep=tiny]
	& {{\scriptstyle (1, 2)}\;\bullet} & {{\scriptstyle (0, 3)}\;\bullet} &&& {{\scriptstyle (1, 3)}\;\bullet} \\
	& {{\scriptstyle (2, 0)}\;\bullet} &&&& {{\scriptstyle (2, 2)}\;\bullet} \\
	{{\scriptstyle (0, 0)}\;\bullet} &&& {{\scriptstyle (1, 4)}\;\bullet} &&& {{\scriptstyle (0, 2)}\;\bullet} \\
	& {{\scriptstyle (2, 1)}\;\bullet} &&&& {{\scriptstyle (3, 0)}\;\bullet} \\
	& {{\scriptstyle (1, 0)}\;\bullet} &&& {{\scriptstyle (0, 1)}\;\bullet} & {{\scriptstyle (1, 1)}\;\bullet}
	\arrow[from=5-2, to=5-5]
	\arrow[from=5-5, to=5-6]
	\arrow[from=5-6, to=3-7]
	\arrow[from=1-6, to=3-7]
	\arrow[from=2-6, to=3-7]
	\arrow[from=2-2, to=3-7]
	\arrow[from=3-4, to=5-5]
	\arrow[from=4-2, to=5-5]
	\arrow[from=1-3, to=3-4]
	\arrow[from=1-3, to=2-6]
	\arrow[from=1-3, to=1-6]
	\arrow[from=1-2, to=1-3]
	\arrow[from=3-1, to=1-2]
	\arrow[from=3-1, to=2-2]
	\arrow[from=3-1, to=4-2]
	\arrow[from=3-1, to=5-2]
	\arrow[from=3-1, to=4-6]
	\arrow[from=4-6, to=3-7]
\end{tikzcd}\]
	while the extended 1\nbd flow graph $\extflow{1}{U}$ is
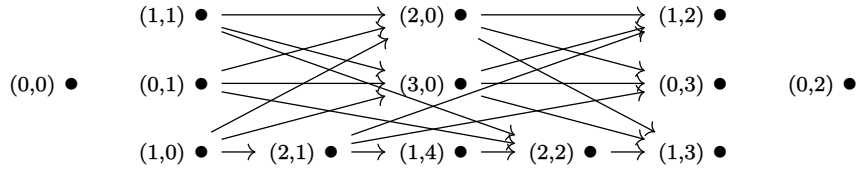
\[\begin{tikzcd}[sep=small]
	& {{\scriptstyle (1, 1)}\;\bullet} && {{\scriptstyle (2, 0)}\;\bullet} && {{\scriptstyle (1, 2)}\;\bullet} \\
	{{\scriptstyle (0, 0)}\;\bullet} & {{\scriptstyle (0, 1)}\;\bullet} && {{\scriptstyle (3, 0)}\;\bullet} && {{\scriptstyle (0, 3)}\;\bullet} & {{\scriptstyle (0, 2)}\;\bullet} \\
	& {{\scriptstyle (1, 0)}\;\bullet} & {{\scriptstyle (2, 1)}\;\bullet} & {{\scriptstyle (1, 4)}\;\bullet} & {{\scriptstyle (2, 2)}\;\bullet} & {{\scriptstyle (1, 3)}\;\bullet}
	\arrow[from=3-2, to=1-4]
	\arrow[from=1-4, to=1-6]
	\arrow[from=1-4, to=2-6]
	\arrow[from=1-4, to=3-6]
	\arrow[from=2-2, to=1-4]
	\arrow[from=1-2, to=1-4]
	\arrow[from=3-2, to=3-3]
	\arrow[from=3-3, to=1-6]
	\arrow[from=3-3, to=2-6]
	\arrow[from=3-3, to=3-4]
	\arrow[from=3-4, to=3-5]
	\arrow[from=2-2, to=3-5]
	\arrow[from=1-2, to=3-5]
	\arrow[from=3-5, to=3-6]
	\arrow[from=1-2, to=2-4]
	\arrow[from=2-2, to=2-4]
	\arrow[from=3-2, to=2-4]
	\arrow[from=2-4, to=1-6]
	\arrow[from=2-4, to=2-6]
	\arrow[from=2-4, to=3-6]
\end{tikzcd}\]
	and the extended 2\nbd flow graph $\extflow{2}{U}$ is
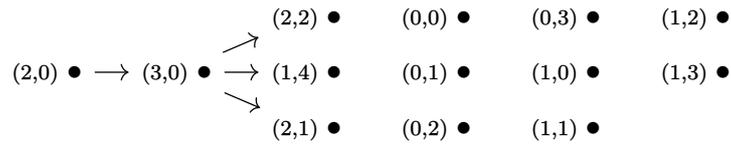
\[\begin{tikzcd}[column sep=small, row sep=tiny]
	&& {{\scriptstyle (2, 2)}\;\bullet} & {{\scriptstyle (0, 0)}\;\bullet} & {{\scriptstyle (0, 3)}\;\bullet} & {{\scriptstyle (1, 2)}\;\bullet} \\
	{{\scriptstyle (2, 0)}\;\bullet} & {{\scriptstyle (3, 0)}\;\bullet} & {{\scriptstyle (1, 4)}\;\bullet} & {{\scriptstyle (0, 1)}\;\bullet} & {{\scriptstyle (1, 0)}\;\bullet} & {{\scriptstyle (1, 3)}\;\bullet} \\
	&& {{\scriptstyle (2, 1)}\;\bullet} & {{\scriptstyle (0, 2)}\;\bullet} & {{\scriptstyle (1, 1)}\;\bullet}
	\arrow[from=2-1, to=2-2]
	\arrow[from=2-2, to=3-3]
	\arrow[from=2-2, to=2-3]
	\arrow[from=2-2, to=1-3]
\end{tikzcd}\]
	all of which are acyclic.
	All other extended flow graphs are discrete, so $U$ is strongly dimension-wise acyclic.
	However, $\hasseo{U}$ contains the cycle
	\[
		(0, 1) \to (1, 1) \to (2, 0) \to (3, 0) \to (2, 1) \to (1, 4) \to (0, 1)
	\]
	so $U$ is not acyclic.
\end{exm}

\begin{prop} \label{prop:hasse_graph_homomorphism}
Let $f\colon P \to Q$ be a morphism of oriented graded posets.
Then $f$ induces a homomorphism of directed graphs $\hasseo{f}\colon \hasseo{P} \to \hasseo{Q}$.
\end{prop}
\begin{proof}
	Let $x, y \in P$ and suppose there is an edge from $x$ to $y$ in $\hasseo{P}$.
	Then either $x \in \faces{}{-}y$, hence $f(x) \in \faces{}{-}f(y)$, or $y \in \faces{}{+} x$, hence $f(y) \in \faces{}{+} f(x)$.
	In either case there is an edge from $f(x)$ to $f(y)$ in $\hasseo{Q}$.
\end{proof}

\begin{lem} \label{lem:morphism_into_acyclic}
Let $f\colon P \to Q$ be a morphism of oriented graded posets.
If $Q$ is acyclic, then $P$ is acyclic.
\end{lem}
\begin{proof}
Suppose that there is a cycle in $\hasseo{P}$.
By Proposition \ref{prop:hasse_graph_homomorphism}, $\hasseo{f}$ maps it onto a cycle in $\hasseo{Q}$.
\end{proof}

\begin{prop} \label{prop:acyclic_has_frame_acyclic_molecules}
	Let $P$ be an acyclic oriented graded poset.
	Then $P$ has frame-acyclic molecules.
\end{prop}
\begin{proof}
	Let $U$ be a molecule and $f\colon U \to P$ be a morphism.
	By Lemma \ref{lem:morphism_into_acyclic}, $U$ is an acyclic regular directed complex, so by Proposition	\ref{prop:acyclicity_implies_strong_dw_acyclicity}, Proposition
	\ref{prop:strong_dw_implies_dw}, and Proposition \ref{prop:dw_acyclic_rdcpx_has_frame_acyclic_molecules}, it is frame-acyclic.
\end{proof}

\begin{cor} \label{cor:acyclic_molecule_hasseo_connected}
	Let $U$ be an acyclic molecule, $x, y \in U$.
	Then there exists a path from $x$ to $y$ or from $y$ to $x$ in $\hasseo{U}$.
\end{cor}
\begin{proof}
	Follows from Proposition \ref{prop:acyclic_has_frame_acyclic_molecules} in combination with Lemma \ref{lem:extflow_connected_frame_acyclic} and Proposition \ref{prop:acyclicity_implies_strong_dw_acyclicity}.
\end{proof}

\begin{cor} \label{cor:acyclic_presents_polygraphs}
	Let $P$ be an acyclic oriented graded poset.
	Then $\molecin{P}$ is a polygraph.
\end{cor}

\begin{prop} \label{prop:morphism_from_molecule_to_acyclic_is_injective}
Let $U$ be a molecule, $P$ an acyclic oriented graded poset, and $f\colon U \to P$ a morphism.
Then $f$ is an inclusion.
\end{prop}
\begin{proof}
	Let $x, y \in U$ and suppose that $f(x) = f(y)$.
	By Corollary \ref{cor:acyclic_molecule_hasseo_connected}, there is a path from $x$ to $y$ or a path from $y$ to $x$ in $\hasseo{U}$.
	Then $\hasseo{f}$ maps this onto a cycle in $\hasseo{P}$, a contradiction, unless $x = y$ and the path is constant.
	We conclude that $f$ is injective.
\end{proof}

\begin{cor} \label{cor:acyclic_omega_cat_basis}
Let $P$ be an acyclic oriented graded poset.
Then
\begin{align*}
	\molecin{P} & = \set{ \isocl{ U \incl P } \mid U \subseteq P, \text{$U$ is a molecule} }, \\
	\atomin{P} & = \set{ \isocl{ \clset{x} \incl P } \mid x \in P, \, \text{$\clset{x}$ is an atom} }.
\end{align*}
\end{cor}
\begin{proof}
	By Proposition \ref{prop:morphism_from_molecule_to_acyclic_is_injective}, every morphism from a molecule to $P$ is an inclusion, equivalent to a subset inclusion $U \incl P$ for some closed subset $U \subseteq P$.
	In particular, every morphism from an atom to $P$ is equivalent to the inclusion $\clset{x} \incl P$ for some $x \in P$.
\end{proof}

\begin{rmk}
	Corollary \ref{cor:acyclic_omega_cat_basis} also implies that $\molecin{P}$ has finitely many cells as soon as $P$ is finite.
\end{rmk}

\begin{rmk}
	Observe that Corollary \ref{cor:acyclic_omega_cat_basis} does not require $P$ to be a regular directed complex, unlike Corollary \ref{cor:strongly_dimensionwise_acyclic_omega_cat}.
\end{rmk}

\noindent
We conclude by showing that acyclic regular directed complexes determine \emph{strong Steiner complexes} in the sense of \cite{ara2020joint}.

\begin{dfn}[Oriented Hasse diagram of an augmented directed chain complex with basis]
	Let $C$ be an augmented directed chain complex with basis $(\grade{n}{\gener{B}})_{n \in \mathbb{N}}$.
	The \emph{oriented Hasse diagram of $C$} is the directed graph $\hasseo{C}$ whose
\begin{itemize}
	\item set of vertices is $\bigcup_{n\in \mathbb{N}} \grade{n}{\gener{B}}$,
	\item set of edges is $\set{ (b, c) \mid \text{$b \in \supp{(\der c)^-}$ or $c \in \supp{(\der b)^+}$} }$, where the source of $(b, c)$ is $b$ and the target is $c$.
\end{itemize}
\end{dfn}

\begin{dfn}[Strong Steiner complex] \index{Steiner complex!strong}
	A \emph{strong Steiner complex} is an augmented directed chain complex $C$ with a unital basis such that $\hasseo{C}$ is acyclic.
	We let $\sstcpx$ denote the full subcategory of $\dchaug$ on strong Steiner complexes.
\end{dfn}

\begin{lem} \label{lem:oriented_hasse_of_dchaug_of_rdcpx}
	Let $P$ be an oriented graded poset such that $\augm{P}$ is oriented thin.
	Then $\hasseo{\dfreeab{P}}$ is isomorphic to $\hasseo{P}$.
\end{lem}
\begin{proof}
	By construction, for all $x \in P$,
	\[
		\supp{(\der x)^+} = \faces{}{+}x, \quad \quad \supp{(\der x)^-} = \faces{}{-}x,
	\]
	so the definitions of $\hasseo{\dfreeab{P}}$ and of $\hasseo{P}$ coincide.
\end{proof}

\begin{prop} \label{prop:acyclic_rdcpx_gives_strong_steiner}
	Let $P$ be an acyclic regular directed complex.
	Then $\dfreeab{P}$ is a strong Steiner complex.
\end{prop}
\begin{proof}
	Follows from Proposition \ref{prop:rdcpx_unital_basis} and Lemma
	\ref{lem:oriented_hasse_of_dchaug_of_rdcpx}.
\end{proof}

\section{Stability under constructions and operations} \label{sec:stability}

In this section, we consider some operations under which the classes of molecules and regular directed complexes are closed --- pastings, suspensions, Gray products, joins, and duals --- and study the stability of acyclicity conditions under these operations.

\begin{dfn}[Suspension of an oriented graded poset]
	Let $P$ be an oriented graded poset.
	The \emph{suspension of $P$} is the oriented graded poset $\sus{P}$ whose
	\begin{itemize}
		\item underlying set is $\set{\sus{x} \mid x \in P} + \set{\bot^+, \bot^-}$,
		\item order and orientation are defined, for all $x \in \sus{P}$ and $\alpha \in \set{+, -}$, by
			\[
			\cofaces{}{\alpha}x \eqdef
			\begin{cases}
				\set{\sus{y} \mid y \in \cofaces{}{\alpha}x'} &
				\text{if $x = \sus{x'}$, $x' \in P$}, \\
				\set{\sus{y} \mid y \in \grade{0}{P}} &
				\text{if $x = \bot^\alpha$}, \\
				\varnothing &
				\text{if $x = \bot^{-\alpha}$}.
			\end{cases}
		\]
	\end{itemize}
\end{dfn}

\begin{dfn}[Gray product of oriented graded posets]
	Let $P$, $Q$ be oriented graded posets.
	The \emph{Gray product} of $P$ and $Q$ is the oriented graded poset $P \gray Q$ whose
	\begin{itemize}
		\item underlying graded poset is the product $P \times Q$ of the underlying posets,
		\item orientation is defined, for all $(x, y) \in P \times Q$ and all $\alpha \in \set{+, -}$, by $\faces{}{\alpha}(x, y) \eqdef \faces{}{\alpha}x \times \set{y} + \set{x} \times \faces{}{(-)^{\dim{x}}\alpha}y$.
	\end{itemize}
	Gray products determine a monoidal structure $(\ogpos, \gray, 1)$ on $\ogpos$.
\end{dfn}

\noindent
The monoidal structure $(\ogpos, \gray, 1)$ restricts to $\ogposbot$, and through the equivalence $\dimin{(-)}$ induces a different monoidal structure on $\ogpos$.

\begin{dfn}[Join of oriented graded posets]
	Let $P$, $Q$ be oriented graded posets.
	The \emph{join of $P$ and $Q$} is the oriented graded poset $P \join Q \eqdef \dimin{(\augm{P}  \gray \augm{Q})}$.
	Joins determine a monoidal structure $(\ogpos, \join, \varnothing)$ on $\ogpos$.
\end{dfn}

\begin{dfn}[Duals of an oriented graded poset]
	Let $P$ be an oriented graded poset, $J \subseteq \posnat$.
	The \emph{$J$\nbd dual of $P$} is the oriented graded poset $\dual{J}{P}$ whose
	\begin{itemize}
		\item underlying set is $\set{ \dual{J}{x} \mid x \in P }$,
		\item partial order and orientation are defined by
			\[
				\faces{}{\alpha} \dual{J}{x} \eqdef
				\begin{cases}
					\set{ \dual{J}{y} \mid y \in \faces{}{-\alpha}x } &
					\text{if $\dim{x} \in J$}, \\
					\set{ \dual{J}{y} \mid y \in \faces{}{\alpha}x } &
					\text{if $\dim{x} \not\in J$}
				\end{cases}
			\]
		for all $x \in P$ and $\alpha \in \set{+,-}$.
	\end{itemize}
	When $J = \posnat$, we write $\optot{P}$ for $\dual{J}{P}$, and call it the \emph{total dual of $P$}.
\end{dfn}

\noindent 
The following collects a number of non-trivial results of \cite[Chapter 7]{hadzihasanovic2024combinatorics}.

\begin{prop} \label{prop:constructions_of_molecules}
	Both the classes of molecules and of regular directed complexes are closed under suspensions, Gray products, joins, and all duals.
\end{prop}

\noindent
We now move on to considering the stability of our acyclicity conditions.

\begin{prop} \label{prop:acyclic_closed_under_pasting}
	Let $U, V$ be molecules and $k \in \mathbb{N}$ such that $U \cp{k} V$ is defined.
	If $U$ and $V$ are acyclic, then $U \cp{k} V$ is acyclic.
\end{prop}
\begin{proof}
	See the proof of \cite[Theorem 2.18]{steiner1993algebra}.
\end{proof}

\begin{exm}
	We show that Proposition \ref{prop:acyclic_closed_under_pasting} does not extend to weaker acyclicity conditions.
	Let $V$ be a 3\nbd dimensional atom whose input and output boundaries correspond to the pasting diagrams
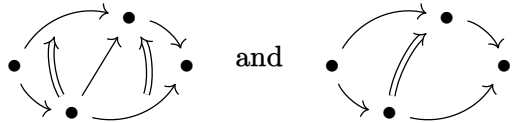
\[\begin{tikzcd}[sep=tiny]
	&& \bullet \\
	\bullet &&& \bullet \\
	& \bullet
	\arrow[""{name=0, anchor=center, inner sep=0}, curve={height=12pt}, from=3-2, to=2-4]
	\arrow[curve={height=6pt}, from=2-1, to=3-2]
	\arrow[curve={height=-6pt}, from=1-3, to=2-4]
	\arrow[""{name=1, anchor=center, inner sep=0}, curve={height=-12pt}, from=2-1, to=1-3]
	\arrow[from=3-2, to=1-3]
	\arrow[curve={height=6pt}, shorten <=7pt, Rightarrow, from=0, to=1-3]
	\arrow[curve={height=-6pt}, shorten >=7pt, Rightarrow, from=3-2, to=1]
\end{tikzcd}
\quad \text{and} \quad 
\begin{tikzcd}[sep=tiny]
	&& \bullet \\
	\bullet &&& \bullet \\
	& \bullet
	\arrow[curve={height=12pt}, from=3-2, to=2-4]
	\arrow[curve={height=6pt}, from=2-1, to=3-2]
	\arrow[curve={height=-6pt}, from=1-3, to=2-4]
	\arrow[curve={height=-12pt}, from=2-1, to=1-3]
	\arrow[curve={height=-6pt}, Rightarrow, from=3-2, to=1-3]
\end{tikzcd}
\]
	respectively, and let $U$ be the 3\nbd dimensional atom from Example 
	\ref{exm:non_acyclic}.
	Then both $V$ and $U$ are strongly dimension-wise acyclic.
	However, the boundary of the pasting $V \cp{2} U$ is isomorphic to the boundary of the 3\nbd dimensional atom from Example \ref{exm:non_dw_acy}, which contains a cycle in its 0\nbd flow graph.
	We conclude that $V \cp{2} U$ is not dimension-wise acyclic.
\end{exm}

\noindent
The following lemma has a straightforward proof.
\begin{lem} \label{lem:flow_graph_of_suspension}
	Let $P$ be an oriented graded poset, $k \in \mathbb{N}$.
	Then
	\begin{enumerate}
		\item $x \mapsto \sus{x}$ induces an isomorphism of directed graphs $\flow{k}{P} \iso \flow{k+1}{\sus{P}}$, restricting to an isomorphism $\maxflow{k}{P} \iso \maxflow{k+1}{\sus{P}}$;
		\item $x \mapsto \sus{x}$ induces an embedding of directed graphs $\extflow{k}{P} \incl \extflow{k+1}{\sus{P}}$, whose complement is the discrete graph on $\{\bot^-, \bot^+\}$.
	\end{enumerate}
\end{lem}

\begin{prop} \label{prop:suspension_of_acyclic}
	Let $P$ be an oriented graded poset.
	If $P$ is acyclic (strongly dimension-wise acyclic, dimension-wise acyclic), then so is $\sus{P}$.
\end{prop}
\begin{proof}
The strongly dimension-wise acyclic and dimension-wise acyclic cases follow from Lemma 
\ref{lem:flow_graph_of_suspension}, together with the observation that $\flow{0}{\sus{P}}$ is always a discrete graph, and that $\extflow{0}{\sus{P}}$ has 
\begin{itemize}
	\item an edge from $\bot^-$ to every element of the form $\sus{x}$,
	\item an edge from every element of the form $\sus{x}$ to $\bot^+$,
\end{itemize}
and no other edges.
The acyclic case is part of \cite[Theorem 2.19]{steiner1993algebra}.
\end{proof}

Next, we prove that frame-acyclicity is also preserved under suspension.

\begin{lem} \label{lem:frame_acyclic_suspension}
Let $U$ be a frame-acyclic molecule.
Then $\sus{U}$ is frame-acyclic.
\end{lem}
\begin{proof}
The following facts are straightforward: the submolecules of $\sus{U}$ are either $\set{\bot^+}$, $\set{\bot^-}$, or $\sus{V}$ for $V \submol U$, and
\[
	\frdim{\sus{V}} = \begin{cases}
		-1 & \text{if $V$ is an atom}, \\
		\frdim{V} + 1 & \text{otherwise}.
	\end{cases}
\]
Given a submolecule $V' \submol \sus{U}$, then, either $V'$ is an atom, in which case $\maxflow{-1}{V'}$ is trivially acyclic, or $V' = \sus{V}$ for $V \submol U$ with $\frdim{V} = r \geq 0$, in which case $\frdim{V'} = r+1$ and, by Lemma \ref{lem:flow_graph_of_suspension}, $\maxflow{r+1}{V'}$ is isomorphic to $\maxflow{r}{V}$, which is acyclic by assumption.
\end{proof}

\begin{prop} \label{prop:rdcpx_frame_acyclic_suspension}
	Let $P$ be a regular directed complex with frame-acyclic molecules.
	Then $\sus{P}$ has frame-acyclic molecules.
\end{prop}
\begin{proof}
	By Lemma \ref{lem:frame_acyclic_suspension}, it suffices to show that every molecule over $\sus{P}$ is, up to isomorphism, either $\set{\bot^\alpha} \incl P$ or of the form $\sus{f}\colon \sus{U} \to \sus{P}$ for some molecule $U$ and morphism $f\colon U \to P$.
	Let $f'\colon U' \to \sus{P}$ be a molecule over $\sus{P}$; we can proceed by induction on submolecules.
	If $U'$ is an atom, by Proposition \ref{prop:rdcpx_local_embeddings} $f'$ is isomorphic to the inclusion $\clset{x} \incl \sus{P}$ for some $x \in \sus{P}$, which is either $\set{\bot^\alpha} \incl \sus{P}$ or $\clset{\sus{x'}} \incl \sus{P}$ for some $x' \in P$, and the latter is isomorphic to $\sus{\clset{x'}} \incl \sus{P}$.
	Otherwise, $f'$ is isomorphic to $g' \cp{k} h'\colon V' \cp{k} W' \to \sus{P}$ with $k < \min\set{\dim{V'}, \dim{W'}}$.
	Then $V'$ and $W'$ are not 0\nbd dimensional, so by the inductive hypothesis $g' = \sus{g}$ and $h' = \sus{h}$ for some $g\colon V \to P$ and $h\colon W \to P$.
	Moreover, for all $\alpha \in \set{+, -}$, necessarily $\bound{0}{\alpha}g' = \bound{0}{\alpha}h' = (\set{\bot^\alpha} \incl \sus{P})$, so $g'$ and $h'$ cannot be 0\nbd composable, and $k > 0$.
	Then $g$ and $h$ are $(k-1)$\nbd composable and $g' \cp{k} h'$ is equal to $\sus{(g \cp{k-1} h)}$ up to isomorphism.
\end{proof}

\begin{prop} \label{prop:gray_join_of_acyclic}
	Let $P$, $Q$ be acyclic oriented graded posets.
	Then $P \gray Q$ and $P \join Q$ are acyclic.
\end{prop}
\begin{proof}
	This is a part of \cite[Theorem 2.19]{steiner1993algebra}.
\end{proof}

\begin{comm}
	This, in conjunction with the results of \cite{ara2020joint} and 
	Theorem \ref{thm:two_omegacats_from_dw_acyclic_rdcpx}, can be used to show that $\molecin{-}$ is compatible with Gray products and joins of strict $\omega$\nbd categories when restricted to acyclic regular directed complexes.
\end{comm}

\begin{exm} \label{exm:dw_acyclic_not_gray_stable}
	We show that strongly dimension-wise acyclic and dimension-wise acyclic molecules are not closed under Gray products.
	Let $U$ be a 3\nbd dimensional atom whose input and output boundary correspond to the pasting diagrams
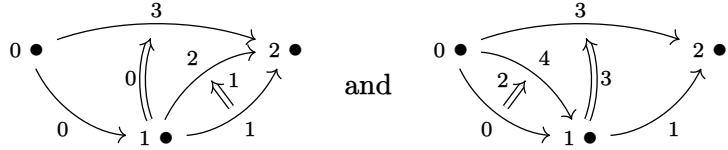
\[\begin{tikzcd}[column sep=small]
	{{\scriptstyle 0}\; \bullet} &&&& {{\scriptstyle 2}\; \bullet} \\
	&& {{\scriptstyle 1}\; \bullet}
	\arrow["0"', curve={height=12pt}, from=1-1, to=2-3]
	\arrow[""{name=0, anchor=center, inner sep=0}, "3", curve={height=-12pt}, from=1-1, to=1-5]
	\arrow[""{name=1, anchor=center, inner sep=0}, "1"', curve={height=12pt}, from=2-3, to=1-5]
	\arrow[""{name=2, anchor=center, inner sep=0}, "2", curve={height=-12pt}, from=2-3, to=1-5]
	\arrow["0", curve={height=-6pt}, shorten >=5pt, Rightarrow, from=2-3, to=0]
	\arrow["1"', shorten <=3pt, shorten >=3pt, Rightarrow, from=1, to=2]
\end{tikzcd}
\quad \text{and} \quad
	\begin{tikzcd}[column sep=small]
	{{\scriptstyle 0}\; \bullet} &&&& {{\scriptstyle 2}\; \bullet} \\
	&& {{\scriptstyle 1}\; \bullet}
	\arrow[""{name=0, anchor=center, inner sep=0}, "0"', curve={height=12pt}, from=1-1, to=2-3]
	\arrow[""{name=1, anchor=center, inner sep=0}, "3", curve={height=-12pt}, from=1-1, to=1-5]
	\arrow[""{name=2, anchor=center, inner sep=0}, "4", curve={height=-12pt}, from=1-1, to=2-3]
	\arrow["1"', curve={height=12pt}, from=2-3, to=1-5]
	\arrow["2", shorten <=3pt, shorten >=3pt, Rightarrow, from=0, to=2]
	\arrow["3"', curve={height=6pt}, shorten >=5pt, Rightarrow, from=2-3, to=1]
\end{tikzcd}\]
	respectively.
	Then $U$ is strongly dimension-wise acyclic.
	However, in $U \gray U$, writing $x \gray y$ instead of $(x, y)$ for better readability, we have
	\begin{align*}
		(0, 1) \gray (2, 2) & \in
		\faces{}{+}((0, 1) \gray (3, 0)) \cap \faces{}{-}((1, 1) \gray (2, 2)), \\
		(1, 1) \gray (1, 0) & \in
		\faces{}{+}((1, 1) \gray (2, 2)) \cap \faces{}{-}((2, 1) \gray (1, 0)), \\
		(2, 1) \gray (0, 1) & \in
		\faces{}{+}((2, 1) \gray (1, 0)) \cap \faces{}{-}((3, 0) \gray (0, 1)), \\
		(2, 2) \gray (0, 1) & \in
		\faces{}{+}((3, 0) \gray (0, 1)) \cap \faces{}{-}((2, 2) \gray (1, 2)), \\
		(1, 4) \gray (1, 2) & \in
		\faces{}{+}((2, 2) \gray (1, 2)) \cap \faces{}{-}((1, 4) \gray (2, 1)), \\
		(0, 1) \gray (2, 1) & \in
		\faces{}{+}((1, 4) \gray (2, 1)) \cap \faces{}{-}((0, 1) \gray (3, 0)).
	\end{align*}
	These relations determine a cycle in $\flow{2}{(U \gray U)}$.
	This proves that $U \gray U$ is not dimension-wise acyclic.
\end{exm}

\begin{dfn}[Converse of a directed graph] \index{directed graph!converse} \index{$\optot{\mathscr{G}}$}
	Let $\mathscr{G}$ be a directed graph.
	The \emph{converse of $\mathscr{G}$} is the directed graph $\optot{\mathscr{G}}$ with
	\begin{itemize}
		\item the same sets of vertices and edges as $\mathscr{G}$,
		\item source and target functions swapped with respect to $\mathscr{G}$.
	\end{itemize}
\end{dfn}

\begin{lem} \label{lem:flow_graphs_under_dual}
	Let $P$ be an oriented graded poset, $J \subseteq \posnat$, $k \geq -1$, and consider the bijection $\dual{J}{}\colon x \mapsto \dual{J}{x}$ between the underlying sets of $P$ and $\dual{J}{P}$.
	Then
	\begin{enumerate}
		\item if $k + 1 \in J$, then $\dual{J}{}$ induces isomorphisms of directed graphs
			\[
				\optot{(\maxflow{k}{P})} \iso \maxflow{k}{\dual{J}{P}}, \quad \optot{(\flow{k}{P})} \iso \flow{k}{\dual{J}{P}}, \quad \optot{(\extflow{k}{P})} \iso \extflow{k}{\dual{J}{P}},
			\]
		\item if $k + 1 \not\in J$, then $\dual{J}{}$ induces isomorphisms of directed graphs
			\[
				\maxflow{k}{P} \iso \maxflow{k}{\dual{J}{P}}, \quad \flow{k}{P} \iso \flow{k}{\dual{J}{P}}, \quad \quad \extflow{k}{P} \iso \extflow{k}{\dual{J}{P}}.
			\]
	\end{enumerate}
\end{lem}

\begin{prop} \label{prop:dw_acyclic_stable_under_all_duals}
	Let $P$ be an oriented graded poset, $J \subseteq \posnat$.
	Then
	\begin{enumerate}
		\item if $P$ is frame-acyclic, then so is $\dual{J}{P}$,
		\item if $P$ is dimension-wise acyclic, then so is $\dual{J}{P}$,
		\item if $P$ is strongly dimension-wise acyclic, then so is $\dual{J}{P}$.
	\end{enumerate}
\end{prop}
\begin{proof}
	Follows from Lemma
	\ref{lem:flow_graphs_under_dual}, combined with the fact that a directed graph is acyclic if and only if its converse is acyclic.
\end{proof}

\begin{exm}
	We show that strongly dimension-wise acyclic and dimension-wise acyclic molecules are not closed under joins.
	Let $U$ be the same 3\nbd dimensional atom as in Example
	\ref{exm:dw_acyclic_not_gray_stable}.
	Since $U$ is strongly dimension-wise acyclic, by Proposition
	\ref{prop:dw_acyclic_stable_under_all_duals} so is its total dual $\optot{U}$.
	Using the isomorphism between $\augm{(U \join \optot{U})}$ and $\augm{U} \gray \augm{(\optot{U})}$, since the total dual counteracts the orientation reversal on faces of the second factor due to dimensions being raised by 1, we see that the cycle in $\flow{2}{(U \gray U)}$ maps to a cycle
	\begin{align*}
		(0, 1) \join \optot{(3, 0)} & \to (1, 1) \join \optot{(2, 2)} \to (2, 1) \join \optot{(1, 0)} \to (3, 0) \join \optot{(0, 1)} \to \\
					    & \to (2, 2) \join \optot{(1, 2)} \to (1, 4) \join \optot{(2, 1)} \to (0, 1) \join \optot{(3, 0)}
	\end{align*}
	in $\flow{3}{(U \join \optot{U})}$.
	This proves that $U \join \optot{U}$ is not dimension-wise acyclic.
\end{exm}

\begin{lem} \label{lem:total_dual_hasseo}
	Let $P$ be an oriented graded poset.
	Then the bijection $x \mapsto \optot{x}$ induces an isomorphism $\optot{(\hasseo{P})} \iso \hasseo{(\optot{P})}$ of directed graphs.
\end{lem}

\begin{prop} \label{prop:acyclic_stable_under_total_dual}
	Let $P$ be an acyclic oriented graded poset.
	Then $\optot{P}$ is acyclic.
\end{prop}
\begin{proof}
	Immediate from Lemma \ref{lem:total_dual_hasseo}.
\end{proof}

\begin{exm}
	Let $U$ be the 3\nbd dimensional atom of Example \ref{exm:non_acyclic}.
	Then $U$ is not acyclic, but $\dual{\set{1}}{U}$ is acyclic.
	Since every dual is involutive, we conclude that acyclicity is not in general stable under duals.
\end{exm}

\bibliographystyle{alpha}
\small \bibliography{main.bib}

\end{document}